\documentclass[reqno,10pt,centertags, draft]{amsart}
\usepackage{hyperref}
\usepackage{mdframed}
\usepackage{amsmath,amsthm,amscd,amssymb,latexsym,esint,upref,stmaryrd,
enumerate,color,verbatim,yfonts,mathrsfs} 

\newcommand*{\mailto}[1]{\href{mailto:#1}{\nolinkurl{#1}}}
\newcommand{\arxiv}[1]{\href{http://arxiv.org/abs/#1}{arXiv:#1}}

\newcommand{\R}{{\bbR}}

\newcommand{\C}{{\mathbb C}}

\newcommand{\bbC}{{\mathbb{C}}}

\newcommand{\bbF}{{\mathbb{F}}}

\newcommand{\bbN}{{\mathbb{N}}}

\newcommand{\bbR}{{\mathbb{R}}}

\newcommand{\cB}{{\mathcal B}}

\newcommand{\cD}{{\mathcal D}}

\newcommand{\cF}{{\mathcal F}}

\newcommand{\cH}{{\mathcal H}}
\newcommand{\cI}{{\mathcal I}}

\newcommand{\cW}{{\mathcal W}}
\newcommand{\cX}{{\mathcal X}}

\newcommand{\gq}{{\mathfrak{q}}}

\DeclareMathOperator{\supp}{supp}

\DeclareMathOperator{\ran}{ran}
\DeclareMathOperator{\dom}{dom}

\DeclareMathOperator*{\slim}{s-lim}

\DeclareMathOperator*{\sgn}{sgn}

\renewcommand{\Re}{\text{\rm Re}}
\renewcommand{\Im}{\text{\rm Im}}
\renewcommand{\ln}{\text{\rm ln}}

\newcommand{\no}{\notag}
\newcommand{\lb}{\label}
\newcommand{\f}{\frac}

\newcommand{\ol}{\overline}

\newcommand{\wti}{\widetilde}
\newcommand{\Oh}{O}
\newcommand{\oh}{o}
\newcommand{\hatt}{\widehat} 
\newcommand{\bi}{\bibitem}
\newcommand{\dott}{\,\cdot\,}

\renewcommand{\ge}{\geqslant}

\let\geq\geqslant
\let\leq\leqslant

\makeatletter
\def\theequation{\@arabic\c@equation}

\allowdisplaybreaks 
\numberwithin{equation}{section}

\newtheorem{theorem}{Theorem}[section]

\newtheorem{lemma}[theorem]{Lemma}

\newtheorem{hypothesis}[theorem]{Hypothesis}
\newtheorem{example}[theorem]{Example}

\theoremstyle{remark}
\newtheorem{remark}[theorem]{Remark}

\begin{document}

\numberwithin{equation}{section}
\allowdisplaybreaks

\title[A Bound for the Eigenvalue Counting Function]{A Bound for the Eigenvalue Counting 
Function For Krein--von Neumann and Friedrichs Extensions} 
  
\author[M.\ S.\ Ashbaugh]{Mark\ S.\ Ashbaugh}
\address{Department of Mathematics, University of
	Missouri, Columbia, MO 65211, USA}
\email{\mailto{ashbaughm@missouri.edu}}
\urladdr{\url{https://www.math.missouri.edu/people/ashbaugh}}
  
\author[F.\ Gesztesy]{Fritz Gesztesy}  
\address{Department of Mathematics,
University of Missouri, Columbia, MO 65211, USA}
\email{\mailto{gesztesyf@missouri.edu}}
\urladdr{\url{https://www.math.missouri.edu/people/gesztesy}}

\author[A.\ Laptev]{Ari Laptev}  
\address{Department of Mathematics, Imperial College London, Huxley Building, 180
Queen�s Gate, London SW7 2AZ, UK}  
\email{\mailto{a.laptev@imperial.ac.uk}}
\urladdr{\url{http://www2.imperial.ac.uk/~alaptev/}}

\author[M.\ Mitrea]{Marius Mitrea}
\address{Department of Mathematics,
University of Missouri, Columbia, MO 65211, USA}
\email{\mailto{mitream@missouri.edu}}
\urladdr{\url{https://www.math.missouri.edu/people/mitream}}

\author[S.\ Sukhtaiev]{Selim Sukhtaiev}
\address{Department of Mathematics, University of
Missouri, Columbia, MO 65211, USA}
\email{\mailto{sswfd@mail.missouri.edu}}

\dedicatory{Dedicated with great pleasure to Yuri Latushkin on the occasion of his 60th birthday.}

\thanks{A.\,L.\ was supported by the RSF grant No 15-11-30007; work of M.\,M.\ was supported by the Simons Foundation Grant $\#$\,281566.} 

\date{\today}

\begin{abstract} 
For an arbitrary open, nonempty, bounded set $\Omega \subset \bbR^n$, $n \in \bbN$, and 
sufficiently smooth coefficients $a,b,q$, we consider the closed, strictly positive, 
higher-order differential operator $A_{\Omega, 2m} (a,b,q)$  in $L^2(\Omega)$ defined 
on $W_0^{2m,2}(\Omega)$, associated with the higher-order differential expression 
$$
\tau_{2m} (a,b,q) := \bigg(\sum_{j,k=1}^{n} (-i \partial_j - b_j) 
a_{j,k} (-i \partial_k - b_k)+q\bigg)^m, \quad  m \in \bbN,
$$ 
and its Krein--von Neumann extension $A_{K, \Omega, 2m} (a,b,q)$ in $L^2(\Omega)$. Denoting by $N(\lambda; A_{K, \Omega, 2m} (a,b,q))$, $\lambda > 0$, the eigenvalue counting function corresponding to the strictly positive eigenvalues of $A_{K, \Omega, 2m} (a,b,q)$, we derive the bound   
$$
N(\lambda; A_{K, \Omega, 2m} (a,b,q)) \leq C v_n (2\pi)^{-n} 
\bigg(1+\frac{2m}{2m+n}\bigg)^{n/(2m)} \lambda^{n/(2m)} , \quad \lambda > 0, 
$$
where $C = C(a,b,q,\Omega)>0$ (with $C(I_n,0,0,\Omega) = |\Omega|$) is connected to the eigenfunction expansion of the self-adjoint operator $\wti A_{2m} (a,b,q)$ in $L^2(\bbR^n)$ defined on 
$W^{2m,2}(\bbR^n)$, corresponding to 
$\tau_{2m} (a,b,q)$. Here $v_n := \pi^{n/2}/\Gamma((n+2)/2)$ denotes the (Euclidean) volume of the unit ball in $\bbR^n$. 

Our method of proof relies on variational considerations exploiting the fundamental link between 
the Krein--von Neumann extension and an underlying abstract buckling problem, and on the 
distorted Fourier transform defined in terms of the eigenfunction transform of $\wti A_{2} (a,b,q)$ in $L^2(\bbR^n)$. 

We also consider the analogous bound for the eigenvalue counting function for the Friedrichs extension $A_{F,\Omega, 2m} (a,b,q)$ in $L^2(\Omega)$ of $A_{\Omega, 2m} (a,b,q)$. 

No assumptions on the boundary $\partial \Omega$ of $\Omega$ are made. 
\end{abstract}

\maketitle 

\newpage 

{\scriptsize{\tableofcontents}}

\section{Introduction}  \lb{s1}

We briefly recall some background material: Suppose $S$ is a densely defined, symmetric, closed operator with nonzero deficiency indices in a separable complex Hilbert space $\cH$ that satisfies 
\begin{equation}
S\geq \varepsilon I_{\cH} \, \text{ for some } \, \varepsilon >0.     \lb{1.1}
\end{equation}
Then, according to M.\ Krein's celebrated 1947 paper \cite{Kr47}, among all nonnegative self-adjoint extensions of $S$, there exist two distinguished ones, $S_F$, the Friedrichs extension of $S$, and  
$S_K$, the Krein--von Neumann extension of $S$, which are, respectively, the largest and 
smallest such extension (in the sense of quadratic forms). In particular, a nonnegative self-adjoint 
operator $\widetilde{S}$ in $\cH$ is a self-adjoint extension of $S$ if and only if $\widetilde{S}$ 
satisfies 
\begin{equation} 
S_K\leq\widetilde{S}\leq S_F  
\end{equation}
(again, in the sense of quadratic forms).

An abstract version of \cite[Proposition\ 1]{Gr83}, presented in \cite{AGMST10}, describing the following intimate connection between the nonzero eigenvalues of $S_K$, and a suitable abstract buckling problem, can be summarized as  follows:
\begin{align}
& \text{There exists $0 \neq v_{\lambda} \in \dom(S_K)$ satisfying } \,  
S_K v_{\lambda} = \lambda v_{\lambda}, \quad \lambda \neq 0,   \lb{1.1a} \\
& \text{if and only if }   \no \\ 
& \text{there exists a $0 \neq u_{\lambda} \in \dom(S^* S)$ such that } \, 
S^* S u_{\lambda} = \lambda S u_{\lambda},   \lb{1.1b} 
\end{align}
and the solutions $v_{\lambda}$ of \eqref{1.1a} are in one-to-one correspondence with the 
solutions $u_{\lambda}$ of \eqref{1.1b} given by the pair of formulas
\begin{equation}
u_{\lambda} = (S_F)^{-1} S_K v_{\lambda},    \quad  
v_{\lambda} = \lambda^{-1} S u_{\lambda}.   \lb{1.1c}
\end{equation}
As briefly recalled in Section \ref{s2}, \eqref{1.1b} represents an abstract buckling problem. 
The latter has been the key in all attempts to date in proving Weyl-type asymptotics for 
eigenvalues of $S_K$ when $S$ represents an elliptic partial differential operator in 
$L^2(\Omega)$. In fact, it is convenient to go one step further and replace the abstract buckling eigenvalue problem \eqref{1.1b} by the variational formulation,  
\begin{align}
\begin{split} 
& \text{there exists $u_{\lambda}\in\dom(S)\backslash\{0\}$ such that}   \\
& \quad \text{${\mathfrak{a}}(w,u_{\lambda})=\lambda\,{\mathfrak{b}}(w,u_{\lambda})$ 
for all $w\in\dom(S)$},  
\end{split} 
\end{align}
where the symmetric forms $\mathfrak{a}$ and $\mathfrak{b}$ in $\cH$ are defined by 
\begin{align}   
\mathfrak{a}(f,g) & :=(Sf,Sg)_{\cH},\quad f,g\in\dom(\mathfrak{a}):=\dom(S),    \\ 
\mathfrak{b}(f,g) & :=(f,Sg)_{\cH},\quad f,g\in\dom(\mathfrak{b}):=\dom(S).    
\end{align}

In our present context, the role of the symmetric operator $S$ will be played by the closed, strictly positive operator in $L^2(\Omega)$,
\begin{equation}
A_{\Omega, 2m} (a,b,q) f = \tau_{2m} (a,b,q) f, \quad 
f \in \dom (A_{\Omega, 2m} (a,b,q)) := W_0^{2m,2}(\Omega),      
\end{equation}
where the differential expression $\tau_{2m} (a,b,q)$ is of the type,
\begin{equation}
\tau_{2m} (a,b,q) := \bigg(\sum_{j,k=1}^{n} (-i \partial_j - b_j) 
a_{j,k} (-i \partial_k - b_k)+q\bigg)^m,  \quad m \in \bbN,
\end{equation}
under the assumption that $\emptyset \neq \Omega \subset \bbR^n$ is open and bounded and 
under sufficient smoothness hypotheses on the coefficients $a,b,q$ (cf.\ Hypothesis \ref{h3.1}\,$(i)$). 
The Krein--von Neumann and Friedrichs extensions of $A_{\Omega, 2m}$ will then be denoted by 
$A_{K, \Omega, 2m} (a,b,q)$ and $A_{F, \Omega, 2m} (a,b,q)$, respectively. 

Since $A_{K,\Omega, 2m} (a,b,q)$ has purely discrete spectrum in $(0,\infty)$ bounded away 
from zero by $\varepsilon > 0$, let $\{\lambda_{K, \Omega, j}\}_{j\in\bbN}\subset(0,\infty)$ be 
the strictly positive eigenvalues of $A_{K,\Omega, 2m} (a,b,q)$ enumerated in nondecreasing order, counting multiplicity, and let
\begin{equation} 
N(\lambda; A_{K,\Omega, 2m} (a,b,q)):=\#\{j\in\bbN\,|\,0<\lambda_{K,\Omega,j} < \lambda\}, 
\quad \lambda > 0, 
\end{equation}
be the eigenvalue distribution function for $A_{K,\Omega, 2m} (a,b,q)$ (which takes into account 
only strictly positive eigenvalues of $A_{K,\Omega, 2m} (a,b,q)$);  
$N(\, \cdot \, ;A_{K,\Omega, 2m} (a,b,q))$ is the principal object of this note. 
Similarly, $N(\lambda; A_{F,\Omega, 2m} (a,b,q))$, $\lambda >0$, denotes the eigenvalue 
counting function for $A_{F,\Omega, 2m} (a,b,q)$. 

For convenience of the reader, we recall the basic abstract facts on the Friedrichs extension, 
$S_F$ and the Krein--von Neumann extension $S_K$ of a strictly positive, closed, symmetric 
operator $S$ in a complex, separable Hilbert space $\cH$ and describe the intimate link between 
the Krein--von Neumann extension and the underlying abstract buckling problem in Section \ref{s2}. Section \ref{s3} focuses on basic domain and spectral properties of the operators, 
$\wti A_{2m} (a,b,q)$, $A_{\Omega, 2m} (a,b,q)$, $A_{K, \Omega, 2m} (a,b,q)$, and 
$A_{F, \Omega, 2m} (a,b,q)$, $m \in \bbN$, and their associated quadratic forms, on open, 
bounded subsets $\Omega \subset \bbR^n$ (without imposing any constraints on $\partial \Omega$). 
In our principal Section \ref{s4} we derive the bounds  
\begin{align}
& N(\lambda; A_{K, \Omega, 2m} (a,b,q)) \leq \frac {v_n} {(2\pi)^n}	
\bigg(1+\frac{2m}{2m+n}\bigg)^{n/(2m)} 
\sup_{\xi\in\R^n}\|\phi(\, \cdot \,, \xi)\|^2_{L^2(\Omega)} \, \lambda^{n/(2m)},     \no \\
& \hspace*{10cm}  \lambda > 0,    \lb{1.11} 
\end{align} 
and 
\begin{align} 
& N(\lambda; A_{F,\Omega,2m} (a,b,q)) \leq \frac {v_n} {(2\pi)^n}	
\bigg(1+\frac{2m}{n}\bigg)^{n/(2m)} 
\sup_{\xi\in\R^n}\|\phi(\, \cdot \,, \xi)\|^2_{L^2(\Omega)} \, \lambda^{n/(2m)},    \no \\
& \hspace*{10cm} \lambda > 0,   \lb{1.12} 
\end{align} 
where $v_n := \pi^{n/2}/\Gamma((n+2)/2)$ denotes the $($Euclidean$)$ volume of the unit ball 
in $\bbR^n$ $($$\Gamma(\cdot)$ being the Gamma function$)$, and $\phi (\, \cdot \, , \, \cdot \,)$ 
represent the suitably normalized generalized eigenfunctions of $\wti A_2 (a,b,q)$ satisfying 
\begin{equation} 
\wti A_2 (a,b,q) \phi(\, \cdot \,,\xi)=|\xi|^2\phi(\, \cdot \,,\xi), \quad \xi \in \bbR^n,
\end{equation} 
in the distributional sense (cf.\ Hypothesis \ref{h4.1}). In particular, whenever the property 
\begin{equation}
\sup_{(x,\xi) \in \Omega \times \bbR^{n}} |\phi(x,\xi)| < \infty  
\end{equation} 
has been established, then 
\begin{equation} 
\sup_{\xi \in \bbR^n} \|\phi(\, \cdot \,, \xi)\|^2_{L^2(\Omega)} \leq |\Omega| 
\sup_{(x,\xi) \in \Omega \times \bbR^{n}} \big(|\phi(x,\xi)|^2\big), 
\end{equation}
explicitly exhibits the volume dependence on $\Omega$ of the right-hand sides of \eqref{1.11} 
and \eqref{1.12}, respectively (see also Section \ref{s5}). 

Our method of employing the eigenfunction transform (i.e., the distorted Fourier transform) 
associated with the variable coefficient operator $\wti A_{2m} (a,b,q)$ (replacing the standard 
Fourier transform in connection with the constant coefficient case in \cite{GLMS15}) to derive 
the results \eqref{1.11} and \eqref{1.12} appears to be new under any assumptions on 
$\partial \Omega$. A comparison of \eqref{1.11}, \eqref{1.12} with the existing literature on 
eigenvalue counting function bounds will be provided in Remark \ref{r4.6}. 

We remark that the power law behavior $\lambda^{n/(2m)}$ coincides with the one in the known 
Weyl asymptotic behavior. This in itself is perhaps not surprising as it is {\it a priori} known that 
\begin{equation}
N(\lambda; A_{K,\Omega, 2m} (a,b,q)) \leq N(\lambda; A_{F,\Omega, 2m} (a,b,q)), 
\quad \lambda > 0,  \lb{1.13}
\end{equation} 
and $N(\lambda; A_{F,\Omega, 2m} (a,b,q))$ is known to have the power law behavior 
$\lambda^{n/(2m)}$ (cf.\ \cite{La97} in the case $a = I_n$, $b=q=0$, extending the corresponding 
result in \cite{LY83} in the case $m=1$). We emphasize that \eqref{1.13} is not in conflict with 
variational eigenvalue estimates since    
$N(\lambda; A_{K,\Omega, 2m} (a,b,q))$ only counts the strictly positive eigenvalues of 
$A_{K,\Omega, 2m} (a,b,q)$ less than $\lambda > 0$ and hence avoids taking into account the (generally, infinite-dimensional) null space of $A_{K,\Omega, 2m} (a,b,q)$. 
Rather than relying on estimates for $N(\, \cdot \, ; A_{F,\Omega, 2m} (a,b,q))$ 
(cf., e.g., \cite{BS70}--\cite{BS80}, 
\cite{Ge14}, \cite{GLW11}, \cite{HH08}, \cite{HH11}, \cite{La97}, \cite{LY83}, \cite{Li80}, 
\cite{Me77}, \cite{NS05}, \cite{Ro71}, \cite{Ro72}, \cite{Ro76}, \cite{Sa01}, \cite{We08}, typically  
for $a = I_n$, $b=0$), we will use the one-to-one correspondence of nonzero eigenvalues 
of $A_{K,\Omega, 2m} (a,b,q)$ with the eigenvalues of its underlying buckling problem 
(cf.\ \eqref{1.1a}--\eqref{1.1c}) and estimate the eigenvalue counting function for the latter. 
Section \ref{s5} illustrates the purely absolutely continuous spectrum and eigenfunction 
assumption we impose on $\wti A_{2m} (a,b,q)$ in $L^2(\bbR^n)$. Finally, 
Appendix \ref{sA} derives a crucial minimization result needed in the derivation of the bound 
\eqref{1.11}, it also compares \eqref{1.11} with the abstract bound \eqref{1.13}, given \eqref{1.12}, 
and points out that the bound \eqref{1.11} is always superior to the abstract one guaranteed by 
combining \eqref{1.12} and \eqref{1.13}. 

In the special case $a=I_n$, $b=q=0$, the bound \eqref{1.11} was derived in \cite{GLMS15}, while 
the bound \eqref{1.12} is due to \cite{La97} in this case. 

Since Weyl asymptotics for $N(\, \cdot \, ; A_{K,\Omega, 2m} (a,b,q))$ and 
$N(\, \cdot \,  ; A_{F,\Omega, 2m} (a,b,q))$ is not considered in this paper (with 
exception of Remark \ref{r4.7}), we just refer to the monographs \cite{Le90} and \cite{SV97}, 
and to \cite{Mi94}, \cite{Mi06}, but note that very detailed bibliographies on this subject 
appeared in \cite{AGMT10} and \cite{AGMST13}. At any rate, the best known result on Weyl asymptotics with remainder estimate for $N(\, \cdot \, ; A_{K,\Omega, 2m} (I_n,0,q))$ to date for bounded Lipschitz domains appears to be \cite{BGMM16} (the case of quasi-convex domains 
having been discussed earlier in \cite{AGMT10}). In contrast to Weyl asymptotics with remainder estimates, the estimates \eqref{1.11}, \eqref{1.12} assume no regularity of $\partial \Omega$ at all. 

We conclude this introduction by summarizing the notation used in this paper. Throughout this 
paper, the symbol $\cH$ is reserved to denote a separable complex Hilbert space with  
$(\dott,\dott)_{\cH}$ the scalar product in $\cH$ (linear in the second argument), and $I_{\cH}$ 
the identity operator in $\cH$. Next, let $T$ be a linear operator mapping (a subspace of) a
Banach space into another, with $\dom(T)$ and $\ran(T)$ denoting the domain and range of $T$. 
The closure of a closable operator $S$ is denoted by $\ol S$. The kernel (null space) of $T$ is 
denoted by $\ker(T)$. The spectrum, point spectrum (i.e., the set of eigenvalues), discrete 
spectrum, essential spectrum, and resolvent set of a closed linear operator in $\cH$ will be 
denoted by $\sigma(\cdot)$, $\sigma_{p}(\cdot)$, $\sigma_{d}(\cdot)$, $\sigma_{ess}(\cdot)$, 
and $\rho(\cdot)$, respectively. The symbol $\slim$ abbreviates the limit in the strong 
(i.e., pointwise) operator topology (we also use this symbol to describe strong limits in $\cH$). 

The Banach spaces of bounded and compact linear operators on $\cH$ are
denoted by $\cB(\cH)$ and $\cB_\infty(\cH)$, respectively. Similarly,
the Schatten--von Neumann (trace) ideals will subsequently be denoted
by $\cB_p(\cH)$, $p\in (0,\infty)$. 
In addition, $U_1\dotplus U_2$ denotes the direct sum of the subspaces $U_1$ 
and $U_2$ of a Banach space $\cX$. Moreover, $\cX_1 \hookrightarrow \cX_2$ denotes the continuous embedding of the Banach space $\cX_1$ into the Banach space $\cX_2$. 

The symbol $L^2(\Omega)$, with $\Omega\subseteq\bbR^n$ open, $n\in\bbN$, 
is a shorthand for $L^2(\Omega,d^n x)$, whenever the $n$-dimensional Lebesgue measure $d^n x$ is 
understood. For brevity, the identity operator in $L^2(\Omega)$ will typically be denoted 
by $I_{\Omega}$. The symbol $\cD(\Omega)$ is reserved for the set 
of test functions $C_0^{\infty}(\Omega)$ on $\Omega$, equipped with the standard 
inductive limit topology, and $\cD'(\Omega)$ represents its dual space, the set of distributions 
in $\Omega$. The distributional pairing, compatible with the $L^2$-scalar product, 
$( \, \cdot \,, \, \cdot \,)_{L^2(\Omega)}$, is abbreviated by 
${}_{\cD'(\Omega)} \langle \, \cdot \, , \, \cdot \, \rangle_{\cD(\Omega)}$. The (Euclidean) volume 
of $\Omega$ is denoted by $|\Omega|$.

The cardinality of a set $M$ is abbreviated by $\# M$. 

For each multi-index $\alpha=(\alpha_1,...,\alpha_n) \in \bbN_0^n$ (abbreviating 
$\bbN_0:=\bbN\cup\{0\}$) we denote by $|\alpha|:=\alpha_1+\cdots+\alpha_n$ the length of $\alpha$.
In addition, we use the standard notations $\partial_j = (\partial/\partial x_j)$, $1 \leq j \leq n$, $\partial^{\alpha} = \partial^{\alpha_1}_{x_1} \cdots \partial^{\alpha_n}_{x_n}$, 
$\nabla = (\partial_{x_1}, \dots , \partial_{x_n})$, and $\Delta = \sum_{j=1}^n \partial^2_j$.

\section{Basic Facts on the Krein--von Neumann 
extension and the Associated Abstract Buckling Problem} 
\lb{s2}

In this preparatory section we recall the basic facts on the Krein--von Neumann extension of 
a strictly positive operator $S$ in a complex, separable Hilbert space $\cH$ and its associated 
abstract buckling problem as discussed in \cite{AGMT10, AGMST10}. 
For an extensive survey of this circle of ideas and an exhaustive list of references 
as well as pertinent historical comments we refer to \cite{AGMST13}.

To set the stage throughout this section, we denote by $S$ a linear, densely defined, symmetric 
(i.e., $S\subseteq S^*$), and closed operator in $\cH$. 
We recall that $S$ is called {\it nonnegative} provided $(f,Sf)_\cH\geq 0$ for all $f\in\dom(S)$.
The operator $S$ is called {\it strictly positive}, if for some $\varepsilon>0$ one has
$(f,Sf)_\cH\geq\varepsilon\|f\|_{\cH}^2$ for all $f\in\dom(S)$; 
one then writes $S\geq\varepsilon I_{\cH}$.  
Next, we recall that two nonnegative, self-adjoint operators $A,B$ in $\cH$ 
satisfy $A\leq B$ (in the sense of forms) if 
\begin{align}\lb{AleqBjussi1}
\dom\big(B^{1/2}\big)\subset\dom\big(A^{1/2}\big)        
\end{align}
and
\begin{align}\lb{AleqBjussi2}
\big(A^{1/2}f,A^{1/2}f\big)_{\cH}\leq\big(B^{1/2}f,B^{1/2}f\big)_{\cH},
\quad f\in\dom\big(B^{1/2}\big). 
\end{align}
We also recall (\cite[Section\ I.6]{Fa75}, \cite[Theorem\ VI.2.21]{Ka80}) that for 
$A$ and $B$ both self-adjoint and nonnegative in $\cH$ one has 
\begin{equation}
0\leq A\leq B\,\text{ if and only if }\,
(B+a I_\cH)^{-1}\leq(A+a I_\cH)^{-1}\,\text{ for all }\,a>0.      
\end{equation}
Moreover, we note the useful fact that $\ker(A)=\ker(A^{1/2})$. 

The following is a fundamental result in M.\ Krein's celebrated 1947 paper
\cite{Kr47} (cf.\ also Theorems~2 and 5--7 in the English summary on page 492): 
 
\begin{theorem}\lb{t2.1}
Assume that $S$ is a densely defined, closed, nonnegative operator in $\cH$. Then, among all 
nonnegative self-adjoint extensions of $S$, there exist two distinguished ones, $S_K$ and $S_F$, 
which are, respectively, the smallest and largest such extension {\rm (}in the sense of 
\eqref{AleqBjussi1}--\eqref{AleqBjussi2}{\rm )}. Furthermore, a nonnegative self-adjoint 
operator $\widetilde{S}$ in $\cH$ is a self-adjoint extension of $S$ if and only if $\widetilde{S}$ 
satisfies 
\begin{equation}\lb{Fr-Sa}
S_K\leq\widetilde{S}\leq S_F.
\end{equation}
In particular, the fact that \eqref{Fr-Sa} holds for all nonnegative self-adjoint extensions 
$\wti S$ of $S$ determines $S_K$ and $S_F$ uniquely. 
In addition,  if $S\geq\varepsilon I_{\cH}$ for some $\varepsilon>0$, one has 
$S_F\geq\varepsilon I_{\cH}$, and 
\begin{align}\lb{SF}
\dom(S_F) & =\dom(S)\dotplus (S_F)^{-1}\ker (S^*),   \\ 
\dom(S_K) & =\dom(S)\dotplus\ker(S^*),    
\lb{SK}   \\
\dom(S^*) & =\dom(S)\dotplus (S_F)^{-1}\ker(S^*)\dotplus\ker(S^*)    \no \\
& =\dom(S_F)\dotplus\ker(S^*),    
\lb{S*} 
\end{align}
and 
\begin{equation}\lb{Fr-4Tf}
\ker(S_K)=\ker\big((S_K)^{1/2}\big)=\ker(S^*)=\ran(S)^{\bot}.
\end{equation} 
\end{theorem}

One calls $S_K$ the {\it Krein--von Neumann extension} of $S$ and $S_F$ the 
{\it Friedrichs extension} of $S$. We also recall that 
\begin{equation}\label{yaf74df}
S_F=S^*|_{\dom(S^*)\cap\dom((S_{F})^{1/2})}.
\end{equation}
Furthermore, if $S\geq\varepsilon I_{\cH}$ for some $\varepsilon > 0$, then \eqref{SK} implies 
\begin{equation}\label{yaere}
\ker(S_K)=\ker\big((S_K)^{1/2}\big)=\ker(S^*)=\ran(S)^{\bot}.
\end{equation} 

For abstract results regarding the parametrization of all nonnegative self-adjoint extensions 
of a given strictly positive, densely defined, symmetric operator we refer the reader to 
Krein \cite{Kr47}, Vi{\v s}ik \cite{Vi63}, Birman \cite{Bi56}, Grubb \cite{Gr68,Gr70}, 
subsequent expositions due to Alonso and Simon \cite{AS80}, Faris \cite[Sect.\ 15]{Fa75}, 
and \cite[Sect.~13.2]{Gr09}, \cite{Gr12}, \cite[Ch.~13]{Sc12}, and Derkach and Malamud \cite{DM91}, 
Malamud \cite{Ma92}, see also \cite[Theorem~9.2]{GM11}.

Let us collect a basic assumption which will be imposed in the rest of this section.

\begin{hypothesis}\lb{h2.2}
Suppose $S$ is a densely defined, symmetric, closed operator with nonzero deficiency 
indices in $\cH$ that satisfies $S\geq\varepsilon I_{\cH}$ for some $\varepsilon >0$. 
\end{hypothesis}

For subsequent purposes we note that under Hypothesis~\ref{h2.2}, one has 
\begin{equation}\lb{jajagutgut}
\dim\big(\ker(S^*-z I_{\cH})\big)=\dim\big(\ker(S^*)\big),
\quad z\in\bbC\backslash[\varepsilon,\infty).  
\end{equation}

We recall that two self-adjoint extensions $S_1$ and $S_2$ of $S$ are called 
{\it relatively prime} (or {\it disjoint}) if $\dom (S_1)\cap\dom (S_2)=\dom (S)$. 
The following result will play a role later on (cf., e.g., \cite[Lemma~2.8]{AGMT10} for an 
elementary proof): 

\begin{lemma}\lb{l2.3}
Assume Hypothesis~\ref{h2.2}. Then the Friedrichs extension $S_F$ and the Krein--von 
Neumann extension $S_K$ of $S$ are relatively prime, that is,
\begin{equation}\label{utrre}
\dom (S_F)\cap\dom(S_K)=\dom(S). 
\end{equation}
\end{lemma}

Next, we consider a self-adjoint operator $T$ in $\cH$ which is bounded from below, 
that is, $T\geq\alpha I_{\cH}$ for some $\alpha\in\bbR$. We denote by 
$\{E_T(\lambda)\}_{\lambda\in\bbR}$ the family of strongly right-continuous spectral 
projections of $T$, and introduce for $-\infty\leq a<b$,  as usual, 
\begin{equation}\label{74ed}
E_T\big((a,b)\big)=E_T(b_{-})-E_T(a)\quad\text{and}\quad 
E_T(b_{-})=\slim_{\varepsilon\downarrow 0}E_T(b-\varepsilon).
\end{equation}  
In addition, we set 
\begin{equation}\label{i5r}
\mu_{T,j}:=\inf\,\big\{\lambda\in\bbR\,\big|\,
\dim(\ran(E_T((-\infty,\lambda))))\geq j\big\},\quad j\in\bbN.
\end{equation} 
Then, for fixed $k\in\bbN$, either: 
\\
$(i)$ $\mu_{T,k}$ is the $k$th eigenvalue of $T$ counting multiplicity 
below the bottom of the essential spectrum, $\sigma_{ess}(T)$, of $T$, 
\\
or, 
\\
$(ii)$ $\mu_{T,k}$ is the bottom of the essential spectrum of $T$, 
\begin{equation}\label{85f4}
\mu_{T,k}=\inf\,\big\{\lambda\in\bbR\,\big|\,\lambda\in\sigma_{ess}(T)\big\}, 
\end{equation}
and in that case $\mu_{T,k+\ell}=\mu_{T,k}$, $\ell\in\bbN$, and there are at 
most $k-1$ eigenvalues (counting multiplicity) of $T$ below $\mu_{T,k}$. 

We now record a basic result of M. Krein \cite{Kr47} with an extension due 
to Alonso and Simon \cite{AS80} and some additional results recently derived in 
\cite{AGMST10}. For this purpose we introduce the {\it reduced  
Krein--von Neumann operator} $\hatt S_K$ in the Hilbert space 
\begin{equation}\lb{hattH}
\hatt\cH:=\big(\ker(S^*)\big)^{\bot}=\big(\ker(S_K)\big)^{\bot}  
\end{equation}
by 
\begin{align}\lb{2.17}
\hatt{S}_K & :=P_{(\ker(S_K))^{\bot}} S_K|_{(\ker(S_K))^{\bot}},  
\quad\dom(\hatt{S}_K)=\dom S_K\cap\hatt\cH,      
\end{align} 
where $P_{(\ker(S_K))^\bot}$ denotes the orthogonal projection onto $(\ker(S_K))^\bot$.
One then obtains  
\begin{equation}\lb{SKinv}
\big(\hatt{S}_K\big)^{-1}=P_{(\ker(S_K))^{\bot}}(S_F)^{-1}|_{(\ker(S_K))^{\bot}},    
\end{equation}
a relation due to Krein \cite[Theorem~26]{Kr47} (see also \cite[Corollary~5]{Ma92}).

\begin{theorem}\lb{t2.4}
Assume Hypothesis~\ref{h2.2}. Then
\begin{equation}\lb{Barr-5}
\varepsilon\leq\mu_{S_F,j}\leq\mu_{\hatt S_K,j},   \quad j\in\bbN.
\end{equation} 
In particular, if the Friedrichs extension $S_F$ of $S$ has purely discrete
spectrum, then, except possibly for $\lambda=0$, the Krein--von Neumann extension
$S_K$ of $S$ also has purely discrete spectrum in $(0,\infty)$, that is, 
\begin{equation}\lb{ESSK}
\sigma_{ess}(S_F)=\emptyset\,\text{ implies }\,\sigma_{ess}(S_K)\subseteq\{0\}.      
\end{equation}
In addition, if $p\in (0,\infty]$, then $(S_F-z_0 I_{\cH})^{-1}\in\cB_p(\cH)$ 
for some $z_0\in\bbC\backslash [\varepsilon,\infty)$ implies
\begin{equation}\lb{CPK}
(S_K-zI_{\cH})^{-1}\big|_{(\ker(S_K))^{\bot}}\in\cB_p\big(\hatt \cH\big) 
\,\text{ for all $z\in\bbC\backslash [\varepsilon,\infty)$}.  
\end{equation}
In fact, the $\ell^p(\bbN)$-based trace ideal $\cB_p(\cH)$ 
$\big($resp., $\cB_p\big(\hatt \cH\big)$$\big)$  
of $\cB(\cH)$ $\big($resp., $\cB\big(\hatt \cH\big)$$\big)$ can be 
replaced by any two-sided symmetrically normed ideal of $\cB(\cH)$ 
$\big($resp., $\cB\big(\hatt \cH\big)$$\big)$.
\end{theorem}

We note that \eqref{ESSK} is a classical result of Krein \cite{Kr47}. Apparently, \eqref{Barr-5} 
in the context of infinite deficiency indices was first proven by Alonso and Simon \cite{AS80} 
by a somewhat different method. The implication \eqref{CPK} was proved in \cite{AGMST10}.

Assuming that $S_F$ has purely discrete spectrum, let 
$\{\lambda_{K, j}\}_{j\in\bbN}\subset(0,\infty)$ be the strictly positive eigenvalues 
of $S_K$ enumerated in nondecreasing order, counting multiplicity, and let
\begin{equation} 
N(\lambda; S_K):=\#\{j\in\bbN\,|\,0<\lambda_{K,j} < \lambda\},  \quad \lambda > 0,    \lb{2.22} 
\end{equation}
be the eigenvalue distribution function for $S_K$. Similarly, let 
$\{\lambda_{F, j}\}_{j\in\bbN}\subset(0,\infty)$ denote the eigenvalues 
of $S_F$, again enumerated in nondecreasing order, counting multiplicity, and by 
\begin{equation} 
N(\lambda; S_F):=\#\{j\in\bbN\,|\, \lambda_{F,j} < \lambda\}, 
\quad \lambda > 0,
\end{equation}  
the corresponding eigenvalue counting function for $S_F$. Then inequality 
\eqref{Barr-5} implies
\begin{equation}
N(\lambda; S_K) \leq N(\lambda; S_F), \quad \lambda > 0.   \lb{2.24} 
\end{equation}
In particular, any upper estimate for the eigenvalue counting function for the Friedrichs extension 
$S_F$, in turn, yields one for the Krein--von Neumann extension $S_K$ 
(focusing on strictly positive eigenvalues of $S_K$ according to \eqref{2.22}). 
While this is a viable approach to estimate the eigenvalue counting function \eqref{2.22} for  
$S_K$, we will proceed along a different route in Section \ref{s3} and directly 
exploit the one-to-one corrspondence between strictly positive eigenvalues of $S_K$ 
and the eigenvalues of its underlying abstract buckling problem to be described next.

To discuss the abstract buckling problem naturally associated with the Krein--von Neumann 
extension as treated in \cite{AGMST10}, we start by introducing an abstract version of 
\cite[Proposition\ 1]{Gr83} (see \cite{AGMST10} for a proof):

\begin{lemma}\lb{l2.5}
Assume Hypothesis~\ref{h2.2} and let $\lambda\in\bbC\backslash\{0\}$. 
Then there exists some $f\in\dom(S_K)\backslash\{0\}$ with
\begin{equation}\lb{sk1}
S_K f=\lambda f,    
\end{equation}
if and only if there exists $w\in\dom(S^* S)\backslash\{0\}$ such that
\begin{equation}\lb{sk2}
S^* Sw=\lambda S w.   
\end{equation}
In fact, the solutions $f$ of \eqref{sk1} are in one-to-one correspondence with the 
solutions $w$ of \eqref{sk2} in the precise sense that 
\begin{align}\label{yt5rr}
w & =(S_F)^{-1}S_K f,   \\ 
f & =\lambda^{-1}Sw.   
\end{align}
\end{lemma}

Of course, since $S_K\geq 0$ is self-adjoint, any $\lambda\in\bbC\backslash\{0\}$ 
in \eqref{sk1} and \eqref{sk2}  necessarily satisfies $\lambda\in(0,\infty)$.

It is the linear pencil eigenvalue problem $S^*Sw=\lambda Sw$ in \eqref{sk2} that we call the 
{\it abstract buckling problem} associated with the Krein--von Neumann extension $S_K$ of $S$.

Next, we turn to a variational formulation of the correspondence between the 
inverse of the reduced Krein--von Neumann extension $\hatt{S}_K$ and the abstract 
buckling problem in terms of appropriate sesquilinear forms by following 
\cite{Ko79}--\cite{Ko84} in the elliptic PDE context. This will then lead to an 
even stronger connection between the Krein--von Neumann extension $S_K$ of $S$ 
and the associated abstract buckling eigenvalue problem \eqref{sk2}, culminating 
in the unitary equivalence result in Theorem~\ref{t2.6} below. 

Given the operator $S$, we introduce the following symmetric forms in $\cH$,
\begin{align}\label{tarcd}
\mathfrak{a}(f,g) & :=(Sf,Sg)_{\cH},\quad f,g\in\dom(\mathfrak{a}):=\dom(S),    \\ 
\mathfrak{b}(f,g) & :=(f,Sg)_{\cH},\quad f,g\in\dom(\mathfrak{b}):=\dom(S).    
\end{align}
Then $S$ being densely defined and closed implies that the sesquilinear form $\mathfrak{a}$ 
shares these properties, while $S\geq\varepsilon I_{\cH}$ from Hypothesis~\ref{h2.2} 
implies that  $\mathfrak{a}$ is bounded from below, specifically, 
\begin{equation}\lb{2.28}
\mathfrak{a}(f,f)\geq\varepsilon^2\|f\|_{\cH}^2,\quad f\in\dom(S).       
\end{equation} 
(Inequality \eqref{2.28} follows from the assumption $S\geq\varepsilon I_{\cH}$ 
by estimating $(Sf,Sf)_{\cH}=\big([(S-\varepsilon I_{\cH})+\varepsilon I_{\cH}]f, 
[(S-\varepsilon I_{\cH})+\varepsilon I_{\cH}]f\big)_{\cH}$ from below.)

Thus, one can introduce the Hilbert space 
\begin{equation}\label{itre3er} 
\cW:=\big(\dom(S),(\, \cdot \,, \, \cdot \,)_{\cW}\big), 
\end{equation} 
with associated scalar product 
\begin{equation} 
(f,g)_{\cW}:=\mathfrak{a}(f,g)=(Sf,Sg)_{\cH},\quad f,g\in\dom(S). 
\end{equation}
In addition, we note that $\iota_{\cW}:\cW\hookrightarrow\cH$, the embedding operator of $\cW$  
into $\cH$, is continuous due to $S\geq\varepsilon I_{\cH}$. Hence, a more precise notation 
would be writing  
\begin{equation}
(w_1,w_2)_{\cW}=\mathfrak{a}(\iota_{\cW} w_1,\iota_{\cW} w_2) 
=(S\iota_{\cW} w_1,S\iota_{\cW} w_2)_{\cH},\quad w_1,w_2\in\cW,   
\end{equation}
but in the interest of simplicity of notation we will omit the embedding 
operator $\iota_{\cW}$ in the following. 

With the sesquilinear forms $\mathfrak a$ and $\mathfrak b$ and the Hilbert space $\cW$ as 
above, given $w_2\in\cW$, the map $\cW\ni w_1\mapsto (w_1,S w_2)_\cH\in{\mathbb{C}}$ is continuous. 
This allows us to define the operator $Tw_2$ as the unique element in $\cW$ with the property that 
\begin{equation}\label{ur332}
(w_1,Tw_2)_{\cW}= (w_1,Sw_2)_{\cH}\,\text{ for all }\,w_1\in\cW.
\end{equation} 
This implies
\begin{equation}\label{oi7g4dc4}
\mathfrak{a}(w_1,Tw_2)=(w_1,Tw_2)_{\cW}=(w_1,Sw_2)_{\cH}=\mathfrak{b}(w_1,w_2) 
\end{equation}
for all $ w_1,w_2\in\cW$. In addition, the operator $T$ satisfies  
\begin{equation}\lb{2.33}
0\leq T=T^*\in\cB(\cW)\quad\text{and}\quad\|T\|_{\cB(\cW)}\leq\varepsilon^{-1}.   
\end{equation}
We will call $T$ the {\it abstract buckling problem operator} associated 
with the Krein--von Neumann extension $S_K$ of $S$.  

Next, recalling the notation $\hatt\cH=\big(\ker(S^*)\big)^{\bot}$ (cf.\ \eqref{hattH}), 
we introduce the operator
\begin{equation}\label{UYBgb}
\hatt{S}:\cW\to\hatt\cH,\quad w\mapsto S w.  
\end{equation}

Clearly, $\ran\big(\hatt{S}\,\big)=\ran(S)$ and since $S\geq\varepsilon I_{\cH}$ 
for some $\varepsilon>0$ and $S$ is closed in $\cH$, $\ran (S)$ is also closed, and 
hence coincides with $\big(\ker(S^*)\big)^\bot$. This yields 
\begin{equation}\label{7hOKI}
\ran\big(\hatt{S}\,\big)=\ran(S)=\hatt\cH.    
\end{equation} 
In fact, it follows that $\hatt{S}\in\cB(\cW,\hatt\cH)$ maps $\cW$ unitarily onto 
$\hatt\cH$ (cf.\ \cite{AGMST10}).

Continuing, we briefly recall the polar decomposition of $S$, 
\begin{equation}\lb{polar}
S=U_S|S|,  
\end{equation} 
where, with $\varepsilon > 0$ as in Hypothesis~\ref{h2.2}, 
\begin{equation}
|S|=(S^*S)^{1/2}\geq\varepsilon I_{\cH}\,\text{ and }\, 
U_S\in\cB\big(\cH,\hatt\cH\big)\,\text{ unitary.}  
\end{equation}

Then the principal unitary equivalence result proved in \cite{AGMST10} reads as follows: 

\begin{theorem}\lb{t2.6}
Assume Hypothesis~\ref{h2.2}. Then the inverse of the reduced Krein--von Neumann extension 
$\hatt S_K$ in $\hatt\cH$ and the abstract buckling problem operator $T$ in $\cW$ are 
unitarily equivalent. Specifically,
\begin{equation}\lb{11.20}
\big(\hatt{S}_K\big)^{-1}=\hatt{S}T\bigl(\hatt{S}\,\bigr)^{-1}.    
\end{equation}
In particular, the nonzero eigenvalues of $S_K$ are reciprocals of the eigenvalues of $T$.  
Moreover, one has
\begin{equation}\lb{11.20a}
\big(\hatt{S}_K\big)^{-1}=U_S\big[|S|^{-1}S|S|^{-1}\big](U_S)^{-1},    
\end{equation}
where $U_S\in\cB\big(\cH,\hatt\cH\big)$ is the unitary operator in the polar 
decomposition \eqref{polar} of $S$ and the operator $|S|^{-1}S|S|^{-1}\in\cB(\cH)$ 
is self-adjoint and strictly positive in $\cH$. 
\end{theorem}

We emphasize that the unitary equivalence in \eqref{11.20} is independent of any spectral assumptions 
on $S_K$ (such as the spectrum of $S_K$ consists of eigenvalues only) and applies to the restrictions 
of $S_K$ to its pure point, absolutely continuous, and singularly continuous spectral subspaces, 
respectively.

Equation \eqref{11.20a} is motivated by rewriting the abstract linear pencil buckling eigenvalue 
problem \eqref{sk2}, $S^*Sw=\lambda Sw$, $\lambda\in\bbC\backslash\{0\}$, in the form
\begin{equation}\label{uu54434}
|S|^{-1}Sw=(S^*S)^{-1/2}Sw=\lambda^{-1}(S^*S)^{1/2}w=\lambda^{-1}|S|w  
\end{equation}
and hence in the form of a standard eigenvalue problem
\begin{equation}\label{8644}
|S|^{-1}S|S|^{-1}v=\lambda^{-1}v,\quad\lambda\in\bbC\backslash\{0\}, 
\quad v:=|S|w.  
\end{equation}
Again, self-adjointness and strict positivity of $|S|^{-1}S|S|^{-1}$ imply $\lambda\in (0,\infty)$. 

We continue this section with an elementary result (recently noted in \cite{GLMS15}) that relates 
the nonzero eigenvalues of $S_K$ directly with the sesquilinear forms $\mathfrak{a}$ and 
$\mathfrak{b}$: 

\begin{lemma} \lb{l2.7}
Assume Hypothesis~\ref{h2.2} and introduce 
\begin{align}\lb{2.42}
& \sigma_p({\mathfrak{a}},{\mathfrak{b}}):=\big\{\lambda\in\bbC\,\big|\,
\text{there exists } \, g_{\lambda}\in\dom(S)\backslash\{0\}     
\nonumber\\ 
& \hspace*{3.00cm} 
\text{such that } \,  
{\mathfrak{a}}(f,g_{\lambda})=\lambda\,{\mathfrak{b}}(f,g_{\lambda}), \quad  f\in\dom(S)\big\}.      
\end{align}
Then
\begin{equation}\lb{2.43}
\sigma_p({\mathfrak{a}},{\mathfrak{b}})=\sigma_p(S_K)\backslash\{0\}     
\end{equation}
{\rm (}counting multiplicity\,{\rm )}, in particular, 
$\sigma_p({\mathfrak{a}},{\mathfrak{b}})\subset(0,\infty)$, and $g_{\lambda}\in\dom(S)\backslash\{0\}$ 
in \eqref{2.42} actually satisfies 
\begin{equation}\lb{2.44}
g_{\lambda}\in\dom(S^*S),\quad S^*Sg_{\lambda}=\lambda Sg_{\lambda}.       
\end{equation}
In addition, 
\begin{equation}\lb{2.45}
\lambda\in\sigma_p({\mathfrak{a}},{\mathfrak{b}})\,\text{ if and only if }\,\lambda^{-1}\in\sigma_p(T)     
\end{equation}
{\rm (}counting multiplicity\,{\rm )}. Finally, 
\begin{equation}\lb{2.46}
T\in\cB_{\infty}(\cW)\,\Longleftrightarrow\,\big(\hatt{S}_K\big)^{-1}\in\cB_{\infty}\big(\hatt\cH\big) 
\,\Longleftrightarrow\,\sigma_{ess}(S_K)\subseteq\{0\}, 
\end{equation}
and hence, 
\begin{equation}\lb{2.47}
\sigma_p({\mathfrak{a}},{\mathfrak{b}})=\sigma(S_K)\backslash\{0\}=\sigma_d(S_K)\backslash\{0\}     
\end{equation}
if \eqref{2.46} holds. In particular, if one of $S_F$ or $|S|$ has purely discrete spectrum 
$($i.e., $\sigma_{ess}(S_F)=\emptyset$ or $\sigma_{ess}(|S|)=\emptyset$$)$, then \eqref{2.46} and 
\eqref{2.47} hold. 
\end{lemma}

One notices that $f\in\dom(S)$ in the definition \eqref{2.42} of 
$\sigma_p({\mathfrak{a}},{\mathfrak{b}})$ can be replaced by $f\in C(S)$ for 
any (operator) core $C(S)$ for $S$ (equivalently, by any form core for the form $\mathfrak{a}$). 

We conclude this section with three auxiliary facts to be used in the proof of Theorem \ref{t4.3} 
and start by recalling an elementary result noted in \cite{GLMS15}: 

\begin{lemma} \lb{l2.8}
	Suppose that $S$ is a densely defined, symmetric, closed operator in $\cH$. Then $|S|$ and hence 
	$S$ is infinitesimally bounded with respect to $S^* S$, more precisely, one has 
	\begin{align}
		\begin{split} 
			\text{for all $\varepsilon > 0$, } \, \|S f\|_{\cB(\cH)} = \| |S| f\|_{\cB(\cH)} \leq 
			\varepsilon \|S^* S f\|_{\cH}^2 + (4 \varepsilon)^{-1} \|f\|_{\cH}^2,&     \lb{3.1} \\
			f \in \dom(S^* S).&         
		\end{split} 
	\end{align}
	In addition, $S$ is relatively compact with respect to $S^* S$ if $|S|$, or equivalently, $S^* S$, has 
	compact resolvent. In particular,
	\begin{equation}
	\sigma_{ess}(S^* S - \lambda S) = \sigma_{ess}(S^* S), \quad \lambda \in \bbR. 
	\end{equation}
\end{lemma}

Given a lower-semibounded, self-adjoint operator $T \geq c_T I_{\cH}$ in $\cH$, we denote 
by $q_T$ its uniquely associated form, that is, 
\begin{equation}
\gq_T(f,g) = \big(|T|^{1/2} f, \sgn(T) |T|^{1/2} g\big)_{\cH},    \quad  
f, g \in \dom(\gq) = \dom \big(|T|^{1/2}\big),  
\end{equation} 
and by $\{E_T(\lambda)\}_{\lambda \in \bbR}$ the family of spectral projections of $T$. 
We recall the following well-known variational characterization of dimensions of spectral 
projections $E_T([c_T, \mu))$, $\mu > c_T$.  

\begin{lemma} \lb{l2.9}
	Assume that $c_T I_{\cH} \leq T$ is self-adjoint in $\cH$ and $\mu > c_T$. Suppose that 
	$\cF \subset \dom \big(|T|^{1/2}\big)$ is a linear subspace such that 
	\begin{equation}
	\gq_T(f,f) < \mu \|f\|_{\cH}^2, \quad f \in \cF\backslash\{0\}.
	\end{equation}  
	Then,
	\begin{equation}
	\dim \big(\ran(E_T([c_T,\mu)))\big) = \sup_{\cF \subset \dom (|T|^{1/2})} (\dim\,(\cF)). 
	\end{equation}
\end{lemma}

We add the following elementary observation: Let 
$c\in\bbR$ and $B\geq c I_{\cH}$ be a self-adjoint operator in $\cH$, and introduce the sesquilinear form $b$ in $\cH$ associated with $B$ via
\begin{align}
	\begin{split}
		& b(u,v) = \big((B - c I_{\cH})^{1/2} u, (B - c I_{\cH})^{1/2} v\big)_{\cH}
		+ c (u,v)_{\cH}, \\  
		& u,v \in \dom(b) = \dom\big(|B|^{1/2}\big).    \lb{B.57}
	\end{split}
\end{align}
Given $B$ and $b$, one introduces the Hilbert space $\cH_b \subseteq \cH$ by 
\begin{align}
	& \cH_b =\big(\dom\big(|B|^{1/2}\big), (\, \cdot \,, \, \cdot \,)_{\cH_b}\big),   \no \\ 
	& (u,v)_{\cH_b} =  b(u,v) + (1-c) (u,v)_{\cH}   \lb{B.58} \\
	& \hspace*{1.2cm} = \big((B - c I_{\cH})^{1/2} u, (B - c I_{\cH})^{1/2} v\big)_{\cH} 
	+ (u,v)_{\cH}    \no \\
	& \hspace*{1.2cm} = \big((B + (1-c) I_{\cH})^{1/2} u, (B + (1-c) I_{\cH})^{1/2} v\big)_{\cH}. 
	\no 
\end{align}
One observes that 
\begin{equation} 
(B + (1 - c)I_{\cH})^{1/2} \colon \cH_b \to \cH \, \text{ is unitary.}     \lb{B.59}
\end{equation}

Finally, we recall the following fact (cf., e.g., \cite{GM09}).

\begin{lemma} \lb{l2.10} 
	Let $\cH$, $B$, $b$, and $\cH_b$ be as in \eqref{B.57}--\eqref{B.59}. Then $B$ has purely discrete spectrum, that is, $\sigma_{ess} (B) = \emptyset$, if and only if $\cH_b$ embeds compactly
	into $\cH$. 
\end{lemma}

\section{Preliminaries on a Class of Partial Differential Operators} 
\lb{s3}

In this section we set the stage for our principal results in Section \ref{s4} and 
introduce the class of even-order partial differential operators $\wti A_{2m} (a,b,q)$ in 
$L^2(\bbR^n)$ as well as $A_{\Omega, 2m} (a,b,q)$ in $L^2(\Omega)$ (see \eqref{taum} for the 
underlying differential expressions), with $\emptyset \neq \Omega \subset \bbR^n$ open 
and bounded (but otherwise arbitrary). In particular, we provide a detailed 
study of their domains and quadratic form domains, including spectral properties such as 
strict boundedness from below for the Friedrichs extension $A_{F,\Omega, 2m} (a,b,q)$ of $A_{\Omega, 2m} (a,b,q)$ in $L^2(\Omega)$, employing a diamagnetic inequality.  

\begin{hypothesis}\lb{h3.1}
$(i)$ Let $m\in\bbN$. Assume that 
\begin{align} 
& b=(b_1,b_2,\dots,b_n)\in \big[W^{(2m-1),\infty}(\R^n)\big]^{n}, \quad b_j \, \text{ real-valued, } 
\, 1 \leq j \leq n,    \lb{f4} \\
& 0 \leq q \in W^{(2m-2),\infty}(\R^n).    \lb{f4a}
\end{align} 
Suppose $a:=\{a_{j,k}\}_{1\leq j,k\leq n}$ is a real symmetric matrix satisfying  
\begin{equation} 
a_{j,k} \in C^{(2m-1)} \big(\R^n\big) \cap L^{\infty}(\bbR^n), \quad 1 \leq j, k \leq n, 
\end{equation} 
and with the property that there exists $\varepsilon_a > 0$ such that
	\begin{equation}\lb{f3}
		\sum_{j,k=1}^{n}a_{j,k}(x) y_j y_k \geq \varepsilon_a |y|^2 
\, \text{ for all } \, x\in\R^n, \; y=(y_1,\dots,y_n) \in \bbR^n.  
	\end{equation}
$(ii)$ Let $\emptyset \neq \Omega \subset \bbR^n$ be open and bounded. 	
In addition, assume that the $n \times n$ matrix-valued function $a$ equals the identity $I_n$ 
outside a ball $B_n(0;R_0)$ containing $\ol \Omega$, that is, there exists $R_0>0$ such that 
\begin{equation}\lb{ajk}
	a(x)=I_n \text{ whenever }  |x| \geq R_0, \text{ and } \,  \ol \Omega\subset B_n(0;R_0).
\end{equation} 
\end{hypothesis}

For simplicity we introduced the ball $B_n(0;R_0)$ containing $\ol \Omega$ in 
Hypothesis \ref{h3.1}\,$(ii)$, but 
for any fixed $\varepsilon > 0$, one can of course replace $B_n(0;R_0)$ by an open $\varepsilon$-neighborhood 
$\Omega_{\varepsilon}$ of $\ol \Omega$.

We will consider various closed (and self-adjoint) $L^2$-realizations of the differential expression   
\begin{align}\lb{taum}
\begin{split} 
& \tau_{2m} (a,b,q) :=\bigg(\sum_{j,k=1}^{n} (-i \partial_j - b_j(x))a_{j,k}(x)
(-i \partial_k - b_k(x))+q(x)\bigg)^m,     \\
& \hspace*{8cm} m \in \bbN, \; x \in \bbR^n.
\end{split} 
\end{align}

We note that Hypothesis \ref{h3.1}\,$(i)$ was of course chosen with $\tau_{2m} (a,b,q)$ in mind. In some instances we only consider the special case $m=1$, that is, $\tau_2 (a,b,q)$, and then choosing the most general case $m=1$ in Hypothesis \ref{h3.1}\,$(i)$ will of course be sufficient. We will tacitly assume such a relaxation of hypotheses on the coefficients $a,b,q$ without necessarily dwelling on this explicitly in every such instance.
 
In the following we find it convenient using auxiliary operators corresponding to the leading and the lower-order terms of the differential expression \eqref{taum}. To this end we first introduce the differential expression $\tau_{2m} (a) = \tau_{2m} (a,0,0)$, 
\begin{equation}
\tau_{2m} (a):= \bigg(-\sum_{j,k=1}^{n} \partial_j a_{j,k}(x) 
\partial_k \bigg)^m,  \quad m \in \bbN, \; x \in \bbR^n,    \lb{Tdifexp}
\end{equation}
and the associated linear operator $\wti T_{2m} (a)$ in $L^2(\R^n)$ given by 
\begin{equation}
\wti T_{2m} (a) u:= \tau_{2m} (a) u, \quad u \in\dom \big(\wti T_{2m} (a)\big): = W^{2m,2}(\R^n).   \lb{Tmin}
\end{equation}
Second, we observe that due to boundedness of the coefficients $a,b,q$ (cf.\ \eqref{f4}) and sufficiently many of their derivatives, one has
\begin{align} \lb{pa}
\begin{split} 
& \tau_{2m} (a,b,q) u = \tau_{2m} (a) u + \sum_{0 \leq |\alpha| \leq 2m-1}g_{\alpha}(a,b,q,x)\partial^{\alpha} u, \\& \tau_{2m} (a,b,q) u\in L^2{(\R^n)}, \quad u \in W^{2m,2}(\R^n), 
\end{split} 
\end{align}
for some $g_{\alpha}(a,b,q, \, \cdot \,)\in L^{\infty}(\R^n)$, $0\leq |\alpha| \leq 2m-1$. The sum of the lower-order terms in \eqref{pa} gives rise to a linear operator $\wti S_{2m-1} (a,b,q)$ in $L^2(\R^n)$, 
\begin{align}
\begin{split} 
& \wti S_{2m-1} (a,b,q) u:=\sum_{0 \leq |\alpha| \leq 2m-1}g_{\alpha}(a,b,q,x)\partial^{\alpha} u, \\ 
& u \in \dom\big(\wti S_{2m-1} (a,b,q)\big):=W^{2m,2}(\R^n).     \lb{3.29n}  
\end{split}
\end{align}

Next, we introduce the operator $\wti A_{2m} (a,b,q)$ in $L^2(\R^n)$ by 
\begin{equation}\lb{k2}
\wti A_{2m} (a,b,q) u:= \tau_{2m} (a,b,q) u, \quad  u \in \dom \big(\wti A_{2m} (a,b,q)\big):=W^{2m,2}(\R^n), 
\end{equation}
and its restriction $\wti A_{0,2m} (a,b,q)$ to $C_0^{\infty}(\R^n)$ in $L^2(\R^n)$ via 
\begin{equation}\lb{g2}
\wti A_{0,2m} (a,b,q) u:=\tau_{2m} (a,b,q) u, \quad 
u \in\dom \big(\wti A_{0,2m} (a,b,q)\big) := C_0^{\infty}(\R^n). 
\end{equation}

Making use of standard perturbation results, it is convenient to view the operator 
$\wti A_{2m} (a,b,q)$ as perturbation of $\wti T_{2m} (a)$ by $\wti S_{2m-1} (a,b,q)$ and state the following auxiliary fact.

\begin{theorem} \lb{t3.2} 
Assume Hypothesis \ref{h3.1}\,$(i)$. Then $\wti A_{0,2m} (a,b,q)$ is essentially self-adjoint in 
$L^2(\bbR^n)$, its closure equals $\wti A_{2m} (a,b,q)$, and hence, 
\begin{equation} 
\wti A_{2m} (a,b,q) \geq 0. 
\end{equation} 
In addition, the graph norm of $\wti A_{2m} (a,b,q)$ is equivalent to the norm of the Sobolev space $W^{2m,2}(\R^n)$, that is, there exist finite constants  $0<c<C$, depending only on $a,b,q,m,n$, such that 
\begin{align} \lb{k3} 
\begin{split}
c\|u\|^2_{W^{2m,2}(\R^n)}\leq \big\|\wti A_{2m} (a,b,q) u \big\|^2_{L^2(\R^n)}+\|u\|^2_{L^2(\R^n)} 
\leq C\|u\|^2_{W^{2m,2}(\R^n)},&    \\ 
u\in W^{2m,2}(\R^n).&  
\end{split} 
\end{align}
\end{theorem}
\begin{proof} We introduce the minimal operator $\wti T_{0,2m} (a)$ in $L^2(\R^n)$ by 
\begin{equation}
\quad\wti T_{0,2m} (a) u:=\tau_{2m} (a) u, \quad u\in\dom(\wti T_{0,2m} (a) ):=C_0^{\infty}(\R^n), 
\end{equation}
and will show that it is essentially self-adjoint and that 
$\wti T_{2m} (a)= \big(\wti T_{0,2m} (a)\big)^*$; the operator $\wti A_{2m} (a,b,q)$ will then be considered as an infinitesimally bounded perturbation of $\wti T_{2m} (a)$. 

Let $u\in L^2(\R^n)\cap W^{2m,2}_{loc}(\R^n)$ and $\tau_{2m} (a) u\in L^2(\R^n)$, then for 
arbitrary $v \in\dom \big(\wti T_{0,2m} (a)\big)= C_0^{\infty}(\R^n)$ one has
\begin{align}
\begin{split} 
& \big(u,\wti T_{0,2m} (a) v\big)_{L^2(\R^n)}=(u,\tau_{2m} (a) v)_{L^2(\R^n)} \\
&\quad={}_{\cD'(\R^n)} \langle \tau_{2m} (a) u, v \rangle_{\cD(\R^n)} 
= (\tau_{2m} (a)u, v)_{L^2(\R^n)}, 
\end{split} 
\end{align}
hence $u\in \dom\big(\big(\wti T_{0,2m} (a)\big)^*\big)$ and 
$\big(\wti T_{0,2m} (a)\big)^*u=\tau_{2m} (a)u$, implying 
\begin{equation}\lb{g3}
\big\{u \in L^2(\R^n) \, \big| \, u\in W^{2m,2}_{loc}(\R^n), \, \tau_{2m} (a) u\in L^2(\R^n)\big\} 
\subseteq \dom\big(\big(\wti T_{0,2m} (a)\big)^*\big).
\end{equation}

Using the interior regularity for elliptic differential operators, one obtains the converse inclusion: 
Indeed, if $u\in \dom\big(\big(\wti T_{0,2m} (a)\big)^*\big)$, then $u\in L^2(\R^n)\subset \cD'(\R^n)$ 
and for some $v \in L^2(\R^n)$ one has $\tau_{2m} (a) u=v$, implying $u\in W^{2m,2}_{loc}(\R^n)$ 
(see, e.g., \cite[Theorem~1.3]{RT05}, see also \cite{Ta11}). 

Our next objective is to show that $\dom\big(\big(\wti T_{0,2m} (a)\big)^*\big)=W^{2m,2}(\R^n)$. Let $\varphi_{R_0}\in C_0^{\infty}(\R^n)$ and $\varphi_{R_0}(x)=1$, $x\in B_n(0;R_0)$, cf.\ \eqref{ajk}. Since $u\varphi_{R_0} \in W^{2m,2}(\R^n)$ for any $u\in \dom\big(\big(\wti T_{0,2m} (a)\big)^*\big)$, in order to prove that $\dom\big(\big(\wti T_{0,2m} (a)\big)^*\big) \subseteq W^{2m,2}(\R^n)$ it suffices to obtain the inclusion  $u(1-\varphi_{R_0})\in W^{2m,2}(\R^n)$. This, in turn, will be guaranteed once we prove the following fact, 
\begin{equation}\lb{k1}
	\dom \big(\big(\wti T_{0,2m} (a)\big)^*(1-\varphi_{R_0})\big) 
	= \dom \big(H_0^m(1-\varphi_{R_0})\big).
\end{equation}
Here the self-adjoint operator $H_0$ in $L^2(\R^n)$ is defined by
\begin{equation}
H_0 u = (- \Delta) u, \quad u \in \dom (H_0) = W^{2,2}(\bbR^n), 
\end{equation}
and hence
\begin{equation}
H_0^{\alpha} u = (- \Delta)^{\alpha} u, \quad 
u \in \dom \big(H_0^{\alpha}\big) = W^{2\alpha,2}(\bbR^n), \quad \alpha \in (0,\infty).   \lb{3.18a} 
\end{equation}

For $u\in\dom (H_0^m(1-\varphi_{R_0}))$, the expression 
$\big(\wti T_{0,2m} (a)\big)^*(1-\varphi_{R_0})u - H_0^m(1-\varphi_{R_0})u$ does not contain derivatives of $u$ of order higher than $2m-1$, therefore, for any $\varepsilon > 0$ there exists some finite $k(\varepsilon)>0$ such that
\begin{align}
\begin{split} 
& \big\|\big(\wti T_{0,2m} (a)\big)^*(1-\varphi_{R_0})u - 
H_0^m(1-\varphi_{R_0})u\big\|_{L^2(\R^n)}^2      \\ 
&\quad \leq \varepsilon \big\|H_0^m(1-\varphi_{R_0})u\big\|_{L^2(\R^n)}^2 
+ k(\varepsilon) \|u\|_{L^2(\R^n)}^2, 
\quad u\in \dom \big(H_0^m(1-\varphi_{R_0})\big).     \lb{g1}
\end{split} 
\end{align}
Combining \eqref{g1} and \cite[Theorem~IV 1.1]{Ka80} one obtains equality of the domains in 
\eqref{k1}, and hence also 
$\dom((\wti T_{0,2m} (a) )^*)\subseteq W^{2m,2}(\R^n)$. The opposite inclusion is clear from \eqref{g3}. 

Next we will show that 
\begin{equation} 
(\wti T_{0,2m} (a) )^*u = \tau_{2m}(a)u, \quad 
u\in\dom\big((\wti T_{0,2m} (a) )^*\big)=W^{2m,2}(\R^n).     \lb{3.22a} 
\end{equation} 
To this end, fix $v \in\dom \big(\wti T_{0,2m} (a)\big) = C_0^{\infty}(\bbR^n)$ and an arbitrary 
$u \in W^{2m,2}(\R^n)$. Then using the membership $\tau_{2m} (a)u\in L^2(\R^n)$, one obtains 
\begin{align} 
\begin{split} 
& \big(u,\wti T_{0,2m} (a)v\big)_{L^2(\R^n)}=(u,\tau_{2m} (a) v)_{L^2(\R^n)}    \\
&\quad={}_{\cD'(\R^n)} \langle \tau_{2m} (a) u, v \rangle_{\cD(\R^n)} 
= (\tau_{2m} (a)u, v)_{L^2(\R^n)},     
\end{split} 
\end{align}
and hence $\big(\wti T_{0,2m} (a)\big)^*u=\tau_{2m} (a)u$. The arbitrariness of $u$ implies that 
$(\wti T_{0,2m} (a) )^*$ is symmetric. Therefore $\wti T_{0,2m} (a)$ is essentially self-adjoint and thus 
$\wti T_{2m} (a) = \big(\wti T_{0,2m} (a)\big)^*$ is self-adjoint. 

The proof thus far showed an important fact: The graph norms of the operators $\wti T_{2m} (a)$ and $H_0^m$, both  defined on $W^{2m,2}(\R^n)$, are equivalent, that is, there exist finite constants $0<c_1<C_1$, depending only on the coefficients $a,b,q,m,n$, such that 
\begin{align}
\begin{split} 
& c_1\big[\big\|H_0^m u\big\|^2_{L^2(\R^n)}+\|u\|^2_{L^2(\R^n)}\big] \leq \big\|\wti T_{2m} (a) u\big\|^2_{L^2(\R^n)}+\|u\|^2_{L^2(\R^n)}    \\
&\quad \leq C_1\big[\big\|H_0^mu\big\|^2_{L^2(\R^n)}+\|u\|^2_{L^2(\R^n)}\big],  \quad 
u \in W^{2m,2}(\R^n).     \lb{g8}
\end{split}
\end{align}
In particular, the graph norm of $\wti T_{2m} (a)$ is equivalent to the norm of $W^{2m,2}(\R^n)$.
	
Finally we show that $\big(\wti A_{0,m} (a,b,q)\big)^*$ is symmetric, actually, self-adjoint, proving that 
$\wti A_{0,m} (a,b,q)$ is essentially self-adjoint. To this end, we recall the operator $\wti S_{2m-1} (a,b,q)$ in \eqref{3.29n}, corresponding to lower-order terms in the differential expression $\tau_{2m} (a,b,q)$. Since $\wti S_{2m-1} (a,b,q)$ has bounded coefficients and its order is at most $2m-1$, it is infinitesimally bounded with respect to the polyharmonic operator $H_0^m$. Thus, for any  
$\varepsilon > 0$ there exists some finite $k(\varepsilon)>0$ such that 
\begin{equation}\lb{g7}
\big\|\wti S_{2m-1} (a,b,q) u \big\|^2_{L^2(\R^n)} \leq \varepsilon \|H_0^mu\|_{L^2(\R^n)}^2 
+ k(\varepsilon)\|u\|_{L^2(\R^n)}^2, \quad u\in W^{2m,2}(\R^n).
\end{equation}
Combining this inequality with the equivalence of the graph norms of $\wti T_{2m} (a)$ and $H_0^m$, one concludes that $\wti S_{2m-1} (a,b,q)$ is infinitesimally bounded with respect to $\wti T_{2m} (a)$. Hence, $\wti A_{0,2m} (a,b,q) =\wti T_{0,2m} (a) + \wti S_{2m-1} (a,b,q)$ is essentially self-adjoint, and 
$\dom ((A_{0,m} (a,b,q))^*)=\dom \big(\wti T_{2m} (a)\big)=W^{2m,2}(\R^n)$. The fact \eqref{k3} follows from \cite[Proposition 7.2]{EE87} and the fact that $\wti A_{2m} (a,b,q)$ and $H_0^m$ have the common domain $W^{2m,2}(\R^n)$ and both are closed (in fact, self-adjoint).
\end{proof}

\begin{lemma}\lb{l3.3}
Assume Hypothesis \ref{h3.1}\,$(i)$. Then for all $\alpha \in (0,1]$, 
\begin{equation}
\dom\big(\big(\wti A_{2m}(a,b,q))^{\alpha}\big) = W^{2m\alpha,2}(\bbR^n),     \lb{3.24a} 
\end{equation}
and there exist finite constants $0<c<C$ depending only on $a,b,q,m,n$, 
such that 
\begin{align}\lb{m6a} 
\begin{split}
c\|u\|^2_{W^{m,2}(\R^n)} \leq \big\|\wti A_{2m} (a,b,q)^{1/2} u\big\|^2_{L^2(\R^n)}+\|u\|^2_{L^2(\R^n)} 
\leq C\|u\|^2_{W^{m,2}(\R^n)},&   \\
u\in W^{m,2}(\R^n),&
\end{split} 
\end{align} 
and hence, 
\begin{align}\lb{m6} 
\begin{split}
c\|u\|^2_{W^{m,2}(\R^n)}\leq \big(u, \wti A_{2m} (a,b,q) u\big)_{L^2(\R^n)}+\|u\|^2_{L^2(\R^n)} 
\leq C\|u\|^2_{W^{m,2}(\R^n)},&   \\
u\in W^{2m,2}(\R^n).&
\end{split} 
\end{align} 
\end{lemma} 
\begin{proof}
We start with a well-known interpolation argument: Let $S$ and $T$ be closed operators in $\cH$  satisfying $\dom(S)\supseteq \dom(T)$. Then $S$ is relatively bounded with respect to $T$ 
(cf., e.g., \cite[Proposition~III.7.2]{EE87},  \cite[Remark~IV.1.5]{Ka80}) and hence
there exist finite constants $a>0$ and $b>0$ such that 
\begin{align}
\||S|f\|_{\cH}^2 = \|Sf\|_{\cH}^2 & \leq a^2 \|Tf\|_{\cH}^2 + b^2 \|f\|_{\cH}^2
= a^2 \||T|f\|_{\cH}^2 + b^2 \|f\|_{\cH}^2  \no \\
& = \big\|\big[a^2 |T|^2 + b^2 I_{\cH}\big]^{1/2} f\big\|_{\cH}^2,
\quad f \in \dom(T)=\dom(|T|).   \lb{3.27a}
\end{align}
Thus, applying the Loewner--Heinz inequality (cf., e.g., 
\cite{Ka52}, \cite[Theorem\ IV.1.11]{KPS82}), one infers that (see also \cite{GLST15}) 
\begin{equation}
\dom\big(|S|^\alpha\big) \supseteq \dom\big(\big(a^2|T|^2 + b^2 I_{\cH}\big)^{\alpha/2}\big)
= \dom\big(|T|^{\alpha}\big), \quad \alpha \in (0,1].   \lb{3.28a}
\end{equation} 
In particular, if $\dom(S) = \dom(T)$ one concludes that 
\begin{equation}
\dom\big(|S|^\alpha\big) = \dom\big(|T|^{\alpha}\big), \quad \alpha \in (0,1].    \lb{3.29a}
\end{equation}
Identifying $S$ with $\wti A_{2m}(a,b,q)$ and $T$ with $H_0^m$, \eqref{3.18a} and \eqref{3.29a} 
prove \eqref{3.24a}.

Employing \eqref{3.24a} with $\alpha = 1/2$ one infers that 
\begin{align}\lb{m7} 
\begin{split} 
\big\|\wti A_{2m} (a,b,q) ^{1/2}u\big\|^2_{L^2(\R^n)}+\|u\|^2_{L^2(\R^n)}  
\approx \big\|H_0^{m/2}u\big\|^2_{L^2(\R^n)}+\|u\|^2_{L^2(\R^n)},&   \\
u\in  W^{m,2}(\R^n).&  
\end{split} 
\end{align}
Assuming, in addition, that $u\in  W^{2m,2}(\R^n)$, the equivalence in \eqref{m7} may be 
rewritten as 
\begin{equation}\lb{m8}
\big(u, \wti A_{2m} (a,b,q) u\big)_{L^2(\R^n)}+\|u\|^2_{L^2(\R^n)} \approx	
\big(u,H_0^mu\big)_{L^2(\R^n)}+\|u\|^2_{L^2(\R^n)},
\end{equation}
and together with the fact that the right-hand side of \eqref{m8} is equivalent to the norm 
$\|\, \cdot \,\|^2_{W^{m,2}(\R^n)}$, one arrives at \eqref{m6}.
\end{proof}

Given Lemma \ref{l3.3}, the sequilinear form $Q_{\wti A_{2m} (a,b,q)}$ in $L^2(\bbR^n)$ associated with $\wti A_{2m}(a,b,q)$ is given by  
\begin{align}\lb{m1}
\begin{split} 
& Q_{\wti A_{2m} (a,b,q)} (u,v):= \big(\wti A_{2m} (a,b,q)^{1/2}u,
\wti A_{2m} (a,b,q)^{1/2}v\big)_{L^2(\R^n)},     \\
& u, v \in \dom (Q_{\wti A_{2m} (a,b,q)}) = \dom \big(\wti A_{2m} (a,b,q)^{1/2}\big) 
= W^{m,2} (\bbR^n), 
\end{split} 
\end{align}
and we also introduce 
\begin{equation}
Q_{H_0^m}(u,v):= \big(H_0^{m/2}u,H_0^{m/2}v\big)_{L^2(\R^n)}, 
\quad u,v \in \dom(Q_{H_0^m}) = W^{m,2}(\R^n).    \lb{m4} 
\end{equation}

In addition, we will employ the explicit representation of the form $Q_{\wti A_{2m} (a,b,q)}$ in terms of $\wti A_{2m} (a,b,q)$, 
\begin{align}\lb{m14} 
& Q_{\wti A_{2m} (a,b,q)} (u,v)= 
\begin{cases} 
(\tau_{2 \ell} u, \tau_{2 \ell} v)_{L^2(\R^n)}, \quad m=2 \ell, \, \ell \in \bbN, \\
\sum_{j,k=1}^n ((- i \partial_j - b_j) \tau_{2 \ell} u, a_{j,k} (- i \partial_k - b_k) \tau_{2 \ell} v)_{L^2(\R^n)} \\ 
+ (\tau_{2 \ell} u,q \, \tau_{2 \ell} v)_{L^2(\R^n)}, \quad m=2 \ell+1, \, \ell \in \bbN \cup \{0\},
\end{cases}     \no \\ 
& \hspace*{8.14cm} u,v\in W^{m,2}(\R^n). 
\end{align} 
Here, in obvious notation, $\tau_0 = 1$.

Assuming Hypothesis \ref{h3.1}\,$(i)$, we introduce one of the main objects of our study, the symmetric operator 
$A_{\Omega,2m} (a,b,q)$ in $L^2(\Omega)$ by 
\begin{equation} \lb{l5.2}
A_{\Omega,2m} (a,b,q) f = \tau_{2m} (a,b,q) f,\quad f\in\dom (A_{\Omega,2m} (a,b,q))=W_0^{2m,2} (\Omega),  
\end{equation}
and note that $\wti A_{2m} (a,b,q)$ formally represents its extended version in $L^2(\R^n)$. In 
addition, we introduce the associated minimal operator $A_{min,\Omega,2m} (a,b,q)$ in 
$L^2(\Omega)$ by 
\begin{align}
\begin{split} 
& A_{min,\Omega,2m} (a,b,q) f:= \tau_{2m} (a,b,q) f,   \\  
& f\in\dom (A_{min,\Omega,2m} (a,b,q)) := C_0^{\infty}(\Omega).    \lb{3.29} 
\end{split} 
\end{align}
Clearly, $A_{min,\Omega,2m} (a,b,q)$ is symmetric (hence, closable) in $L^2(\Omega)$ (upon elementary integration by parts) and nonnegative, 
\begin{equation}
A_{min,\Omega,2m} (a,b,q) \geq 0.   \lb{3.30} 
\end{equation} 

\begin{theorem} \lb{t3.5} 
Assume Hypothesis \ref{h3.1}\,$(i)$. Then the closure of 
$A_{min,\Omega,2m} (a,b,q)$ in $L^2(\Omega)$ is given by $A_{\Omega,2m} (a,b,q)$, 
\begin{equation}
\ol{A_{min,\Omega,2m} (a,b,q)} = A_{\Omega,2m} (a,b,q).    \lb{l5.2a}
\end{equation}
In particular, $A_{\Omega,2m} (a,b,q)$ is symmetric and nonnegative in $L^2(\Omega)$,
\begin{equation}
A_{\Omega,2m} (a,b,q) \geq 0.    \lb{3.32a} 
\end{equation} 
In addition, there exist finite constants $0<c<C$, depending only on $a,b,q,m,n$, such that
\begin{align}
\begin{split} 
c \|f\|_{W_0^{2m,2}(\Omega)}^2 \leq \|A_{\Omega,2m} (a,b,q)f\|^2_{ L^2(\Omega)} 
+ \|f\|^2_{L^2(\Omega)} \leq C \|f\|_{W_0^{2m,2}(\Omega)}^2,& \\ 
f \in W_0^{2m,2} (\Omega).&       \lb{l5.15} 
\end{split} 
\end{align}  
\end{theorem}
\begin{proof} Using \eqref{k3} with $v\in C_0^{\infty} (\R^n)$, $\supp(v) \subset \Omega$ one 
concludes that the graph norm of $A_{\Omega,2m} (a,b,q)$ is equivalent to the norm of 
$\mathring{W}^{2m,2}(\Omega)$ on $C_0^{\infty} (\Omega)$. Therefore, 
$\dom\big(\ol{A_{min,\Omega,2m} (a,b,q)}\big) = \mathring{W}^{2m,2}(\Omega)$. In order to prove that
$\ol{A_{min,\Omega,2m} (a,b,q)} f = \tau_{2m} (a,b,q) f$, we consider $\{f_j\}_{j \in\bbN} \subset C_0^{\infty}(\Omega)$, $f, g \in L^2(\Omega)$, 
such that 
\begin{equation} \lb{15.13}
\lim_{j \to \infty} \|f_j - f\|_{L^2(\Omega)} = 0 \, \text{ and } \,  
\lim_{j \to \infty} \big\|A_{min,\Omega,2m} (a,b,q) f_j - g\big\|_{L^2(\Omega)} = 0.
\end{equation} 
Since $A_{min,\Omega,2m} (a,b,q)$ is symmetric and hence closable in $L^2(\Omega)$, one infers that 
\begin{equation} 
f\in\dom\big(\ol{A_{min,\Omega,2m} (a,b,q)}\big) = W_0^{2m,2} (\Omega) \, \text{ and } \,  
\ol{A_{min,\Omega,2m} (a,b,q)} f = g. 
\end{equation} 
Taking arbitrary $\psi \in C_0^{\infty}(\Omega)$, and recalling our notation for the distributional pairing 
${}_{\cD'(\Omega)} \langle \, \cdot \,, \, \cdot \, \rangle_{\cD(\Omega)} $ (compatible with the scalar product $(\, \cdot \,, \, \cdot \,)_{L^2(\Omega)}$), one concludes that 
\begin{align}
& (g, \psi)_{L^2(\Omega)} = {}_{\cD'(\Omega)} \langle g, \psi \rangle_{\cD(\Omega)} 
= \lim_{j \to \infty} {}_{\cD'(\Omega)} \langle \tau_{2m} (a,b,q) f_j , \psi \rangle_{\cD(\Omega)} \no \\ 
& \quad = \lim_{j\to \infty} \int_{\Omega} \ol{f_j(x)} \, \big(\tau_{2m} (a,b,q) \psi\big)(x) \, d^n x = \int_{\Omega} \ol{f(x)} \, \big(\tau_{2m} (a,b,q) \psi\big)(x) \, d^n x    \no \\
& \quad =  {}_{\cD'(\Omega)} \langle \tau_{2m} (a,b,q) f , \psi \rangle_{\cD(\Omega)},    
\end{align}
implying $g = \tau_{2m} (a,b,q) f$ and hence,  
$\ol{A_{min,\Omega,2m} (a,b,q)} f = A_{\Omega,2m} (a,b,q) f$ implying \eqref{l5.2a}. This also completes the proof of 
\eqref{l5.15}. Finally, being the closure of the symmetric operator $A_{min,\Omega,2m} (a,b,q)$, also 
$A_{\Omega,2m} (a,b,q)$ is symmetric in $L^2(\Omega)$ (cf., e.g., \cite[Theorem~5.4\,(b)]{We80}). 
\end{proof}

Next, still assuming Hypothesis \ref{h3.1}\,$(i)$, we introduce the form 
$Q_{A_{\Omega,2m} (a,b,q)}$ in $L^2(\Omega)$ 
generated by $A_{\Omega,2m} (a,b,q)$, via 
\begin{align}
\begin{split}
& Q_{A_{\Omega,2m} (a,b,q)}(f,g):=(f,A_{\Omega,2m} (a,b,q) g)_{L^2(\R^n)},    \\
& f, g \in \dom(Q_{A_{\Omega,2m} (a,b,q)}) := W_0^{2m,2}(\Omega).     \lb{m4.1} 
\end{split} 
\end{align}

\begin{lemma}\lb{l3.6} 
Assume Hypothesis \ref{h3.1}\,$(i)$. 
Then the form $Q_{A_{\Omega,2m} (a,b,q)}$  is closable and its closure in $L^2(\Omega)$, denoted 
by $Q_{A_{F,\Omega,2m} (a,b,q)}$, is the form uniquely associated to the Friedrichs extension 
$A_{F,\Omega,2m} (a,b,q)$ of $A_{\Omega,2m} (a,b,q)$, that is, 
\begin{align}\lb{m10} 
\begin{split}
& Q_{A_{F,\Omega,2m} (a,b,q)} (f,g) 
= \big(A_{F,\Omega,2m} (a,b,q)^{1/2} f, A_{F,\Omega,2m} (a,b,q)^{1/2} g\big)_{L^2(\Omega)},  \\
& f, g \in \dom \big(Q_{A_{F,\Omega,2m} (a,b,q)}) = \dom \big(A_{F,\Omega,2m} (a,b,q)^{1/2}\big) 
= W_0^{m,2}(\Omega).
\end{split} 
\end{align}
\end{lemma}
\begin{proof}
That $Q_{A_{\Omega,2m} (a,b,q)}$ is closable follows from abstract results relating sectorial (in particular, non-negative, symmetric) operators and their forms (cf., e.g., 
\cite[Theorem~IV.2.3]{EE87}, \cite[Theorem~VI.1.27]{Ka80}, \cite[Theorem~X.23]{RS75}). 
In order to prove \eqref{m10}, we fix $f \in W_0^{2m,2}(\Omega)$ and denote its extension 
by zero outside of $\Omega$ by $\wti f$. Then $\wti f \in W^{2m,2}(\R^n)$ and employing \eqref{m6} with $u$ replaced by $\wti f$, and using the fact that supp$(\wti u)\subseteq \Omega$, one obtains 
\begin{equation}\lb{m11}
c\|f\|^2_{W_0^{m,2}(\Omega)} \leq (f,A_{\Omega,2m} (a,b,q) f)_{L^2(\Omega)} 
+ \|f\|^2_{L^2(\Omega)} \leq C\|f\|^2_{W_0^{m,2}(\Omega)},
\end{equation}
that is, 
\begin{equation}\lb{m12}
c\|f\|^2_{W_0^{m,2}(\Omega)} \leq Q_{A_{\Omega,2m}(a,b,q)}(f,f) + \|f\|^2_{L^2(\Omega)} 
\leq C \|f\|^2_{W_0^{m,2}(\Omega)},
\end{equation}
for some finite constants $0<c<C$, independent of $f$, proving that the domain of the closure of the form 
$Q_{A_{\Omega,2m} (a,b,q)}$ equals 
$W_0^{m,2}(\Omega)$. Together with \cite[Sect.~VI.2.3]{Ka80} or 
\cite[Theorem~X.23]{RS75}, and the second representation theorem for forms (see, e.g., 
\cite[Theorem IV.2.6, Theorem IV.2.8] {EE87}, \cite[Theorem~VI.2.23]{Ka80}), this proves \eqref{m10}.  
\end{proof}

In Section \ref{s4} we will also use the following explicit representation of the form 
$Q_{A_{F,\Omega,2m} (a,b,q)}$, 
\begin{align}\lb{m15} 
& Q_{A_{F,\Omega,2m} (a,b,q)} (f,g)= 
\begin{cases} 
(\tau_{2 \ell} f,\tau_{2 \ell} g)_{L^2(\Omega)}, \quad m=2 \ell, \, \ell \in \bbN, \\
\sum_{j,k=1}^n ((- i \partial_j - b_j) \tau_{2 \ell} f, a_{j,k} (-i \partial_k - b_k) \tau_{2 \ell} g)_{L^2(\Omega)} \\
+ (\tau_{2 \ell} f, q \, \tau_{2 \ell} g)_{L^2(\Omega)}, \quad m=2 \ell+1, \, k \in \bbN \cup \{0\},  
\end{cases}      \no \\
& \hspace*{8.65cm} f,g\in W_0^{m,2}(\Omega).   
\end{align} 
(Again, we use the convention $\tau_0 = 1$.)

Finally, we introduce the following symmetric form in $L^2(\Omega)$, 
\begin{align}\lb{l5.16} 
\begin{split}
&\mathfrak{a}_{\Omega,4m,a,b,q} (f,g) 
:=(A_{\Omega, 2m} (a,b,q) f,A_{\Omega, 2m} (a,b,q) g)_{L^2(\Omega)}, \\
&f,g\in\dom(\mathfrak{a}_{\Omega, 4m,a,b,q}):=\dom(A_{\Omega, 2m} (a,b,q)),   
\end{split} 
\end{align}
and the Hilbert space 
\begin{align}
\begin{split} 
\cH_{A_{\Omega,2m} (a,b,q)}
:=& (\dom( A_{\Omega,2m} (a,b,q)), \mathfrak{a}_{\Omega,4m,a,b,q} (\, \cdot \,,\, \cdot \,))  \\
=&\big (W_0^{2m,2} (\Omega), \mathfrak{a}_{\Omega,4m,a,b,q} (\, \cdot \,,\, \cdot \,)\big), 
\lb{3.52a}
\end{split}
\end{align} 
equipped with the scalar product $ \mathfrak{a}_{\Omega,4m,a,b,q} (\, \cdot \,,\, \cdot \,)$.

\begin{lemma}\lb{l3.7}
Assume Hypothesis \ref{h3.1}\,$(i)$. Then the Hilbert space $\cH_{A_{\Omega,2m} (a,b,q)}$ embeds 
compactly into $L^2(\Omega)$.
\end{lemma}
\begin{proof}
This is a consequence of the compact embedding of $\mathring{W}^{2m,2}(\Omega)$ into $L^2(\Omega)$ (see, e.g., \cite[Theorem~V.4.18]{EE87}) and the inequalities \eqref{l5.15}.
\end{proof}

At this point we strengthen the lower bounds \eqref{3.30}, \eqref{3.32a}:

\begin{theorem}\lb{t3.8}
Assume Hypothesis \ref{h3.1}\,$(i)$ with $m=1$. Then there exists 
$\varepsilon > 0$, depending only on $a$ and $\Omega$, such that  
$A_{min,\Omega,2} (a,b,q)$ defined as in \eqref{3.29} with $m=1$ satisfies 
\begin{equation}\lb{3.44a} 
A_{min,\Omega,2} (a,b,q) \geq \varepsilon I_{\Omega},
\end{equation} 
and hence, 
\begin{equation}\lb{3.45a} 
A_{\Omega,2} (a,b,q) \geq \varepsilon I_{\Omega} \, \text{ and } \, 
A_{F,\Omega,2} (a,b,q) \geq \varepsilon I_{\Omega}.
\end{equation} 
\end{theorem}
\begin{proof}
It suffices to prove that there 
exists $\varepsilon>0$ such that $A_{F,\Omega,2} (a,b,q) \geq \varepsilon I_{\Omega}$. Since 
$\dom \big(A_{F,\Omega,2} (a,b,q)^{1/2}\big) = W_0^{1,2}(\Omega)$ according to 
\eqref{m10}, one recalls that  
\begin{equation}
f \in  W_0^{1,2}(\Omega) \, \text{ implies } \, |f| \in W_0^{1,2}(\Omega)
\end{equation}
(cf., e.g., \cite[Corollary~VI.2.4]{EE87}), and that by \cite[Proposition~4.4]{Ou05}, 
\begin{equation}
\partial_j |f| = \Re\big(\sgn\big(\ol f\big) (\partial_j f)\big)  
\text{ a.e.,} \quad f \in W_0^{1,2}(\Omega), 
\; 1 \leq j \leq n,
\end{equation}
with 
\begin{equation}
\sgn (g(x)) = \begin{cases} g(x)/|g(x)|, & \text{if $g(x) \neq 0$,} \\
0, & \text{if $g(x) = 0$.}   \end{cases} 
\end{equation}
Thus, $\nabla |f| = \Re \big(\sgn\big(\ol f\big) (\nabla f)\big)$, $ f \in  W_0^{1,2}(\Omega)$, and 
hence one obtains the diamagnetic inequality on $\Omega$, 
\begin{align}
& |\nabla |f|| \leq \big|\Re \big(\sgn\big(\ol f\big) (\nabla f)\big)\big| 
= \big|\Re \big(\sgn\big(\ol f\big) ((\nabla - i b) f)\big)\big| \leq |(- i \nabla - b) f| \text{ a.e.,}  \no \\
& \hspace*{9cm} f \in  W_0^{1,2}(\Omega), 
\end{align} 
since $b_j$, $1 \leq j \leq n$, are real-valued, according to a device of Kato \cite{Ka72} and Simon \cite{Si76} (see also \cite[Theorem~4.5.1]{BE11}, \cite[Theorem~7.21]{LL01}). Hence, 
employing the min-max principle for the infimum of the spectrum of self-adjoint operators bounded 
from below one estimates,
\begin{align}
& \inf (\sigma(A_{F,\Omega,2} (a,b,q))) 
= \inf_{f \in W_0^{1,2}(\Omega), \, \|f\|_{L^2(\Omega)}=1} 
Q_{A_{F,\Omega,2} (a,b,q)} (f, f)   \no \\
& \quad = \inf_{f \in W_0^{1,2}(\Omega), \, \|f\|_{L^2(\Omega)}=1} 
\big(A_{F,\Omega,2} (a,b,q)^{1/2} f, A_{F,\Omega,2} (a,b,q)^{1/2} f\big)_{L^2(\Omega)} 
\no \\
& \quad = \inf_{f \in W_0^{1,2}(\Omega), \, \|f\|_{L^2(\Omega)}=1} 
\bigg(\sum_{j,k=1}^n ((- i \partial_j - b_j) f, a_{j,k} (-i \partial_k - b_k) f)_{L^2(\Omega)} 
\no \\
& \hspace*{8.5cm} + (f, q \, f)_{L^2(\Omega)}\bigg)   \no \\
& \quad \geq \varepsilon_a 
 \inf_{f \in W_0^{1,2}(\Omega), \, \|f\|_{L^2(\Omega)}=1} 
((- i \nabla - b) f, (-i \nabla - b) f)_{L^2(\Omega)^n}   \no \\
& \quad \geq \varepsilon_a 
 \inf_{f \in W_0^{1,2}(\Omega), \, \|f\|_{L^2(\Omega)}=1} 
(|\nabla |f||, |\nabla |f||)_{L^2(\Omega)}   \no \\
& \quad = \varepsilon_a 
 \inf_{f \in W_0^{1,2}(\Omega), \, \|f\|_{L^2(\Omega)}=1} 
(\nabla |f|, \nabla |f|)_{L^2(\Omega)^n}   \no \\
& \quad \geq \varepsilon_a  \inf_{\varphi \in W_0^{1,2}(\Omega), \, \|\varphi\|_{L^2(\Omega)}=1} 
\big((\nabla \varphi, \nabla \varphi)_{L^2(\Omega)^n}\big)   \no \\
& \quad \geq \varepsilon_a \inf \big(\sigma\big(- \Delta^D_{\Omega}\big)\big)    \no \\
& \quad = \varepsilon_a \varepsilon_{\Omega} =: \varepsilon,
\end{align}
using the fact that $- \Delta^D_{\Omega} \geq \varepsilon_{\Omega} I_{\Omega}$ for some 
$\varepsilon_{\Omega} > 0$, since $\Omega$ is bounded (see, for instance, \cite[p.~31]{Da89}, 
or use domain monotonicity, \cite[p.~270]{RS78} together with the well-known strictly positive lower bounds for a ball or cube that encloses $\Omega$). 
\end{proof}

The result \eqref{3.45a} holds under more general assumptions on the coeffcients $a,b,q$ and 
also for certain boundary conditions other than Dirichlet, but the current setup suffices for our 
discussion in Section \ref{s4} (we intend to revisit this issue elsewhere).  

Next, we note that as a consequence of Hypothesis \ref{h3.1}\,$(i)$, also all higher-order powers  
$A_{\Omega,2m} (a,b,q) = A_{\Omega,2} (a,b,q)^m$, $m \in \bbN$, $m \geq 2$, of $A_{\Omega,2} (a,b,q)$ are strictly positive.  

\begin{lemma} \lb{l3.8}
Assume Hypothesis \ref{h3.1}\,$(i)$. Then there exists $\varepsilon_m >0$ such that 
\begin{equation}
A_{\Omega,2m} (a,b,q) \geq \varepsilon_m I_{\Omega}, \quad m \in \bbN.
\end{equation}
\end{lemma}
\begin{proof}
We employ induction with respect to $m \in \bbN$. The case $m=1$ holds by 
Hypothesis \ref{h3.1}\,$(i)$. Assume that the statement holds for all $k<m$ and fix any 
$0\not=f\in\dom (A_{\Omega,2m} (a,b,q))$. We consider two cases: \\[1mm] 
$(i)$ $m=2 \ell$, $\ell\in\bbN$. Then due to symmetry of $A_{\Omega,2}(a,b,q)^{\ell}$ one obtains
\begin{equation}\lb{b1}
(f,A_{\Omega,2m} (a,b,q) f)_{L^2({\Omega})} = \big(f,A_{\Omega,2}(a,b,q)^{2\ell} f\big)_{L^2({\Omega})} 
= \big\|A_{\Omega,2}(a,b,q)^{\ell} f\big\|^2_{L^2({\Omega})}. 
\end{equation}
By the induction hypothesis, there exists $\varepsilon_\ell>0$
such that, $A_{\Omega,2\ell} (a,b,q)\geq \varepsilon_{\ell}$, and hence 
\begin{equation}\lb{b3}
\varepsilon_{\ell} \|f\|_{L^2({\Omega})}^2\leq (f,A_{\Omega,2\ell} (a,b,q) f)_{L^2({\Omega})}
\leq \|f\|_{L^2({\Omega})} \|A_{\Omega,2\ell} (a,b,q) f\|_{L^2({\Omega})},
\end{equation}
implying 
\begin{align}
(f,A_{\Omega,2m} (a,b,q) f)_{L^2({\Omega})} 
= \big\|A_{\Omega,2}(a,b,q)^{\ell} f\big\|^2_{L^2({\Omega})} \geq 
\varepsilon_{\ell}^2 \|f\|^2_{L^2({\Omega})} = \varepsilon_{m} \|f\|^2_{L^2({\Omega})},
\end{align}
with $\varepsilon_m = \varepsilon_{\ell}^2$. \\[1mm] 
$(ii)$ $m=2\ell+1$, $\ell\in\bbN$. Then by \eqref{3.44a}
\begin{align}
&(f,A_{\Omega,2m} (a,b,q)f)_{L^2({\Omega})}
= \big(f,A_{\Omega,2}(a,b,q)^{2\ell+1}f\big)_{L^2({\Omega})} 
\no \\
& \quad = \big(A_{\Omega,2}(a,b,q)^{\ell}f, A_{\Omega,2}(a,b,q) 
A_{\Omega,2}(a,b,q)^{\ell}f\big)_{L^2({\Omega})}   \no\\ 
&\quad \geq \varepsilon \big\|A_{\Omega,2}(a,b,q)^{\ell}f\big\|^2_{L^2({\Omega})} 
\geq \varepsilon\varepsilon_{\ell}^2 \|f\|^2_{L^2({\Omega})} = \varepsilon_m \|f\|^2_{L^2({\Omega})},\lb{b2}
\end{align} 
with $\varepsilon_m = \varepsilon \varepsilon_{\ell}^2$. 
\end{proof}

\section{An Upper Bound for the Eigenvalue Counting Function for the Krein--von Neumann and Friedrichs  Extensions of Higher-Order Operators} \lb{s4} 

In this section we derive an upper bound for the eigenvalue counting function for Krein--von Neumann  extensions of higher-order differential operators on open, bounded, nonempty domains $\Omega \subset \bbR^n$. In particular, no assumptions on the boundary of $\Omega$ will be made.

In the following we denote by 
$A_{K, \Omega, 2m} (a,b,q)$ and $A_{F, \Omega, 2m} (a,b,q)$ the Krein--von Neumann and 
Friedrichs extensions of $A_{\Omega,2m} (a,b,q)$ in $L^2(\Omega)$. Since by 
Lemma \ref{l3.7}, $\cH_{A_{\Omega,2m} (a,b,q)}$ embeds 
compactly into $L^2(\Omega)$, $A_{\Omega,2m} (a,b,q)^* A_{\Omega,2m} (a,b,q)$ has purely discrete spectrum by 
Lemma \ref{l2.10}. Equivalently, $A_{\Omega,2m} (a,b,q)^* A_{\Omega,2m} (a,b,q)$ has a 
compact resolvent, in particular, 
\begin{equation} 
[A_{\Omega,2m} (a,b,q)^* A_{\Omega,2m} (a,b,q)]^{-1} \in \cB_{\infty}\big(L^2(\Omega)\big). 
\end{equation} 
Consequenty, also
\begin{equation} 
|A_{\Omega,2m} (a,b,q)|^{-1} = [A_{\Omega,2m} (a,b,q)^* A_{\Omega,2m} (a,b,q)]^{-1/2} \in 
\cB_{\infty}\big(L^2(\Omega)\big), 
\end{equation}       
implying 
\begin{equation}
\big(\hatt A_{K, \Omega, 2m} (a,b,q)\big)^{-1} \in \cB_{\infty}\big(L^2(\Omega)\big) 
\end{equation}  
by \eqref{11.20a}. Thus,  
\begin{equation}
\sigma_{ess}(A_{K, \Omega, 2m} (a,b,q)) \subseteq \{0\}.   
\end{equation} 
We recall that the form $\mathfrak{a}_{\Omega,4m,a,b,q}(\, \cdot\, , \, \cdot \,)$ in $L^2(\Omega)$ associated with the operator $A_{\Omega,2m} (a,b,q)^* A_{\Omega,2m} (a,b,q)$ has been introduced in \eqref{l5.16}. 

Let $\{\lambda_{K, \Omega, j}\}_{j\in\bbN}\subset(0,\infty)$ be the strictly positive eigenvalues 
of $A_{K, \Omega, 2m} (a,b,q)$ enumerated in nondecreasing order, counting multiplicity, and let
\begin{equation}\label{4455}
N(\lambda; A_{K, \Omega, 2m} (a,b,q)):=\#\{j\in\bbN\,|\,0<\lambda_{K,\Omega,j} < \lambda\}, \quad 
\lambda > 0, 
\end{equation}
be the eigenvalue distribution function for $A_{K, \Omega, 2m} (a,b,q)$. 

To derive an effective estimate for $N(\lambda; A_{K, \Omega, 2m} (a,b,q))$ we need to introduce 
one more spectral hypothesis imposed on $\wti A_{2m} (a,b,q)$:

\begin{hypothesis} \lb{h4.1}
Assume Hypothesis \ref{h3.1}. \\
$(i)$ Suppose there exists $\phi:\R^n\times \R^n\rightarrow \C$ such that the operator 
\begin{equation} 
(\bbF f)(\xi):=(2\pi)^{- n/2}\int_{\R^n}f(x)\overline{\phi(x,\xi)} \,d^nx, \quad \xi\in \R^n, 
\end{equation} 
originally defined on functions $f\in L^2(\R^n)$ with compact support, can be extended to a unitary operator in $L^2(\R^n)$, such that 
\begin{equation}\lb{k9}
f \in W^{2,2}(\R^n; d^n x) \, \text{ if and only if } \, |\xi|^2(\bbF f)(\xi)\in L^2(\R^n; d^n \xi),
\end{equation} 
and 
\begin{equation}\lb{k10}
\wti A_{2} (a,b,q)=\bbF^{-1}M_{|\xi|^2}\bbF,
\end{equation}
where $M_{|\xi|^2}$ represents the maximally defined operator of multiplication by $|\xi|^2$ in 
$L^2(\R^n; d^n \xi)$. \\[1mm] 
$(ii)$ In addition, suppose that 
\begin{equation}
\lb{gef}
\sup_{\xi\in \R^n}\|\phi(\, \cdot \,, \xi)\|_{L^2(\Omega)}<\infty. 
\end{equation} 
\end{hypothesis}

\begin{remark} \lb{r4.2} 
$(i)$ As becomes clear from Theorems \ref{t4.3} and \ref{t4.4} below, our primary concerns are the operators $A_{K, \Omega, 2m} (a,b,q)$ and $A_{F, \Omega, 2m} (a,b,q)$ in $L^2(\Omega)$, and hence we are primarily interested  in the coefficients $a,b,q$ on the open, bounded, but otherwise arbitrary, set $\Omega$. However, since the existence of an eigenfunction expansion of a self-adjoint ``continuation'' of this pair of operators to all of $\bbR^n$, denoted by $\wti A_{2m} (a,b,q)$, is a 
crucial tool in our derivation of the bound on the corresponding eigenvalue counting functions of 
$A_{K, \Omega, 2m} (a,b,q)$ and $A_{F, \Omega, 2m} (a,b,q)$, the continuation of the coefficients 
$a,b,q$ through a possibly highly nontrivial boundary $\partial \Omega$ of $\Omega$ becomes a 
nontrivial issue. To avoid intricate technicalities, we chose to simply assume a sufficiently smooth behavior of $a,b,q$ throughout $\bbR^n$ in Hypothesis \ref{h3.1}\,$(i)$. \\[1mm] 
$(ii)$ Hypothesis \ref{h4.1}\,$(i)$ implies that  $\wti A_{2} (a,b,q)$ (and hence any of its powers) 
is spectrally purely absolutely continuous (i.e., its point and singular continuous spectra are 
empty), while Hypothesis \ref{h4.1}\,$(ii)$ requires a uniform $L^2(\Omega)$-bound on 
$\phi(\, \cdot \,, \xi)$, $\xi \in \bbR^n$. In particular, $\phi(\, \cdot \,, \, \cdot \,)$ represent the suitably normalized generalized eigenfunctions of $\wti A_2 (a,b,q)$ satisfying 
\begin{equation}
\wti A_2 (a,b,q) \phi(\, \cdot \,,\xi)=|\xi|^2\phi(\, \cdot \,,\xi),  \quad \xi \in \bbR^n,
\end{equation} 
in the distributional sense. In the special Laplacian case $a = I_n$, $b= q = 0$, one obtains
\begin{equation} 
\phi(x, \xi) = e^{i \xi \cdot x}, \quad 
\|\phi(\, \cdot \,, \xi)\|^2_{L^2(\Omega)} = |\Omega|, \quad (x, \xi) \in \bbR^{2n}.
\end{equation}
$(iii)$ In the case $a = I_n$, and with the exception of a possible zero-energy resonance 
and/or eigenvalue of $\wti A_{2}(I_n,b,q)$ in $L^2(\bbR^n)$, we 
expect Hypothesis \ref{h4.1} to hold for $\wti A_{2} (I_n,b,q)$ under the regularity assumptions 
made on $b,q$ in Hypothesis \ref{h3.1}\,$(i)$ assuming in addition that $b$ and $q$ 
have sufficiently fast decay as $|x| \to \infty$ (e.g., if  $b, q$  have compact support). While 
we have not found the corresponding statement in the literature, and an attempt to prove it in full generality would be an independent project, we will illustrate in our final Section \ref{s5} explicit situations in which Hypothesis \ref{h4.1} holds for $a = I_n$. The case $a \neq I_n$, on the other hand,  is much more involved due trapping/non-trapping issues which affect the existence of bounds of the type \eqref{egs}; we refer, for instance, to \cite{Bu01}, \cite{Bu02}, \cite{DDZ15}, \cite{Vo00}, 
\cite{Vo00a}, and the literature therein. \\[1mm]
$(iv)$ We note from the outset, that a zero-energy resonance and/or eigenvalue of $\wti A_{2m}$ cannot be excluded even in the special case $a = I_n$, $b=0$, and $q \in C_0^{\infty}(\bbR^n)$. However, the existence of such zero-energy resonances or eigenvalues is highly unstable with 
respect to small variations of $a,b,q$ and hence their absence holds generically. In particular, 
by slightly varying $R_0  > 0$ in Hypothesis \ref{h3.1}\,$(ii)$, or the $\varepsilon$-neighborhood 
$\Omega_{\varepsilon}$ of $\Omega$ mentioned after \eqref{ajk}, or by slightly perturbing the coefficients $a$, $b$, or $q$ outside $B_n(0; R_0)$, or outside $\Omega_{\varepsilon}$, one can guarantee the absence of such zero-energy resonances and/or eigenvalues. Since we are primarily interested in the operators $A_{K, \Omega, 2m} (a,b,q)$ and $A_{F, \Omega, 2m} (a,b,q)$ in $L^2(\Omega)$, we can indeed freely choose the form of $a, b, q$ in an $\varepsilon$-neighborhood outside of $\Omega$, especially, in a neighborhood of infinity. ${}$ \hfill $\diamond$ 
\end{remark}

With the standard notation 
\begin{equation}
x_+ := \max \, (0, x), \quad x \in \bbR,
\end{equation}
we have the following estimate for $N(\, \cdot \, ;A_{K, \Omega, 2m} (a,b,q))$ (extending the 
results in \cite{La97} where the special case $a = I_n$, $b=q=0$ has been considered): 

\begin{theorem} \lb{t4.3}
Assume Hypothesis \ref{h4.1}. Then the following estimate holds:
\begin{align}
& N(\lambda; A_{K, \Omega, 2m} (a,b,q)) \leq \frac {v_n} {(2\pi)^n}	
\bigg(1+\frac{2m}{2m+n}\bigg)^{n/(2m)} \sup_{\xi\in\R^n}\|\phi(\, \cdot \,, \xi)\|^2_{L^2(\Omega)} \, 
\lambda^{n/(2m)},     \no \\
& \hspace*{8.7cm} \, \text{ for all } \,  \lambda > 0,    \lb{3.36} 
\end{align} 
where $v_n := \pi^{n/2}/\Gamma((n+2)/2)$ denotes the $($Euclidean$)$ volume of the unit ball in 
$\bbR^n$ $($$\Gamma(\cdot)$ being the Gamma function$)$, and $\phi (\, \cdot \, , \, \cdot \,)$ 
represents the suitably normalized generalized eigenfunctions of $\wti A_2 (a,b,q)$ satisfying 
$\wti A_2 (a,b,q) \phi(\, \cdot \,,\xi)=|\xi|^2\phi(\, \cdot \,,\xi)$ in the distributional sense 
$($cf.\ Hypothesis \ref{h4.1}$)$. 
\end{theorem}
\begin{proof}
Following our abstract Section \ref{s2}, we introduce in addition to the symmetric form 
$\mathfrak{a}_{\Omega,4m,a,b,q}$ in $L^2(\Omega)$ (cf.\ \eqref{l5.16}), the form 
\begin{align} 
\begin{split}  
& \mathfrak{b}_{\Omega,2m,a,b,q}(f,g) :=(f,A_{\Omega, 2m} (a,b,q)g)_{L^2(\Omega)}, \\ 
& f,g\in\dom(\mathfrak{b}_{\Omega,2m,a,b,q}):=\dom(A_{\Omega, 2m} (a,b,q)).    
\end{split} 
\end{align}
By Lemma \ref{l2.7}, particularly, by \eqref{2.45}, one concludes that 
\begin{align}
\begin{split} 
& N(\lambda; A_{K, \Omega, 2m} (a,b,q)) \leq \max \big(\dim\, \big\{f \in \dom(A_{\Omega, 2m} (a,b,q)) \,\big|    \\
& \hspace*{3.2cm} 
\mathfrak{a}_{\Omega,4m,a,b,q}(f,f) 
- \lambda \, \mathfrak{b}_{\Omega,2m,a,b,q}(f,f) < 0\big\}\big).   \lb{3.32} 
 \end{split} 
\end{align}
Here we also employed \eqref{2.47} and the fact that 
\begin{align}
\begin{split} 
& \mathfrak{a}_{\Omega,4m,a,b,q}(f_{K,\Omega,j},f_{K,\Omega,j})) - \lambda \, 
\mathfrak{b}_{\Omega,2m,a,b,q}(f_{K,\Omega,j},f_{K,\Omega,j})    \\ 
& \quad = (\lambda_{K,\Omega,j} - \lambda) \|f_{K,\Omega,j}\|_{L^2(\Omega)}^2 < 0,
\end{split} 
\end{align}
where $f_{K,\Omega,j} \in \dom(A_{\Omega, 2m} (a,b,q)) \backslash \{0\}$ additionally satisfies 
\begin{align} 
\begin{split} 
& f_{K,\Omega,j} \in \dom(A_{\Omega, 2m} (a,b,q)^* A_{\Omega, 2m} (a,b,q)) \, \text{ and } \\ 
&A_{\Omega, 2m} (a,b,q)^* A_{\Omega, 2m} (a,b,q) f_{K,\Omega,j} 
= \lambda_{K,\Omega,j} \, A_{\Omega, 2m} (a,b,q) f_{K,\Omega,j}. 
\end{split} 
\end{align}
To further analyze \eqref{3.32} we now fix $\lambda \in (0,\infty)$ and introduce the auxiliary operator 
\begin{align}
\begin{split}
& L_{\Omega, 4m, \lambda} (a,b,q) := A_{\Omega, 2m} (a,b,q)^* A_{\Omega, 2m} (a,b,q) 
- \lambda \, A_{\Omega, 2m} (a,b,q), \\
& \dom(L_{\Omega, 4m, \lambda} (a,b,q)) := \dom(A_{\Omega, 2m} (a,b,q)^* A_{\Omega, 2m} (a,b,q)). 
\end{split} 
\end{align}
By Lemma \ref{l2.8}, $L_{\Omega, 4m, \lambda} (a,b,q)$ is self-adjoint, bounded from below, with purely 
discrete spectrum as its form domain satisfies (cf.\ \eqref{3.52a})
\begin{equation} 
\dom\big(|L_{\Omega, 4m, \lambda} (a,b,q)|^{1/2}\big) = \cH_{A_{\Omega, 2m} (a,b,q)}, 
\end{equation}  
and the latter embeds compactly into $L^2(\Omega)$ by Lemma \ref{l3.7} (cf.\ Lemma \ref{l2.10}). 
We will study the auxiliary eigenvalue problem,
\begin{equation}
L_{\Omega, 4m, \lambda} (a,b,q) \varphi_j = \mu_j \varphi_j, \quad 
\varphi_j \in \dom(L_{\Omega, 4m, \lambda} (a,b,q)), 
\end{equation}
where $\{\varphi_j\}_{j \in \bbN}$ represents an orthonormal basis of eigenfunctions in 
$L^2(\Omega)$ and for simplicity of notation we repeat the eigenvalues $\mu_j$ of 
$L_{\Omega, 4m, \lambda} (a,b,q)$ according to their multiplicity. Since 
$\varphi_j \in W_0^{2m,2} (\Omega)$, the zero-extension of $\varphi_j$ to all of $\bbR^n$, 
\begin{equation}
\wti \varphi_j (x) := \begin{cases} \varphi_j(x), & x \in \Omega, \\ 0, & 
x \in \bbR^n \backslash \Omega,  \end{cases}
\end{equation} 
satisfies  
\begin{equation} 
\wti \varphi_j \in W^{2m,2} (\bbR^n), \quad 
\partial^{\alpha} \wti \varphi_j = \wti{\partial^{\alpha} \varphi_j}, \quad 0 \leq |\alpha| \leq 2m. 
\end{equation} 

Next, given $\mu > 0$, one estimates
\begin{align} 
\begin{split} 
\mu^{-1} \sum_{\substack{j \in \bbN \\ \mu_j < \mu}} (\mu - \mu_j) & \geq 
\mu^{-1} \sum_{\substack{j \in \bbN, \\ \mu_j < 0, \, \mu_j < \mu}} (\mu - \mu_j) \geq 
\mu^{-1} \sum_{\substack{j \in \bbN, \\ \mu_j < 0, \, \mu_j < \mu}} \mu   \\
& = n_-(L_{\Omega, 4m, \lambda} (a,b,q)), 
\end{split} 
\end{align}
where $n_-(L_{\Omega, 4m, \lambda} (a,b,q))$ denotes the number of strictly negative eigenvalues of $L_{\Omega, 4m,\lambda} (a,b,q)$. Combining, Lemma \ref{l2.9} and \eqref{3.32} one concludes that 
\begin{align}
& N(\lambda; A_{K, \Omega, 2m} (a,b,q)) 
\leq \max \big(\dim\, \big\{f \in \dom(A_{min, \Omega, 2m}) \,\big|    \no \\ 
& \hspace*{2,8cm}
 \mathfrak{a}_{\Omega,4m, a,b,q}(f,f) - \lambda \, \mathfrak{b}_{\Omega,2m,a,b,q}(f,f) < 0\big\}\big)  
 \lb{3.44}  \\ 
& \quad = n_-(L_{\Omega, 4m, \lambda} (a,b,q)) \leq 
\mu^{-1} \sum_{\substack{j \in \bbN \\ \mu_j < \mu}} (\mu - \mu_j) 
= \mu^{-1} \sum_{j \in \bbN} [\mu - \mu_j]_+, \quad \mu > 0.     \no 
\end{align} 

Next, we focus on estimating the right-hand side of \eqref{3.44}. 
\begin{align}
& N(\lambda; A_{K, \Omega, 2m} (a,b,q)) \leq \mu^{-1} \sum_{j \in \bbN} (\mu - \mu_j)_+ 
= \mu^{-1} \sum_{j \in \bbN} \big[(\varphi_j,(\mu - \mu_j) \varphi_j)_{L^2(\Omega)}\big]_+ \no \\   
&\quad = \mu^{-1} \sum_{j \in \bbN} \Big[\mu \|\varphi_j\|_{L^2(\Omega)}^2  
- \|A_{\Omega,2m} (a,b,q)\varphi_j\|_{L^2(\Omega)}^2    \no \\
&  \hspace*{3.3cm} 
+\lambda (\varphi_j, A_{\Omega,2m} (a,b,q) \varphi_j)_{L^2(\Omega)}\Big]_+   \no \\ 
&\quad = \mu^{-1} \sum_{j \in \bbN} \Big[\mu \|\wti \varphi_j\|_{L^2(\bbR^n)}^2  
-\big\|\wti A_{2m} (a,b,q) \wti \varphi_j\big\|_{L^2(\bbR^n)}^2     \no \\ 
& \hspace*{3.4cm} 
+\lambda  \big(\wti \varphi_j, \wti A_{2m} (a,b,q) \wti \varphi_j\big)_{L^2(\bbR^n)}\Big]_+   \no \\
&\quad = \mu^{-1} \sum_{j \in \bbN} \bigg[\int_{\bbR^n} 
\big[\mu -|\xi|^{4m} +\lambda|\xi|^{2m}\big] 
|(\bbF{\wti \varphi}_j)(\xi)|^2 \, d^n \xi\bigg]_+   \no \\
&\quad \leq \mu^{-1} \sum_{j \in \bbN} \int_{\bbR^n} 
\big[\mu -|\xi|^{4m} +\lambda|\xi|^{2m}\big]_+  
|(\bbF{\wti \varphi}_j)(\xi)|^2 \, d^n \xi    \no \\
&\quad  \leq \mu^{-1} \int_{\bbR^n} 
\big[\mu -|\xi|^{4m} +\lambda|\xi|^{2m}\big]_+  
\sum_{j \in \bbN} |(\bbF{\wti \varphi}_j)(\xi)|^2 \, d^n \xi. \lb{3.45} 
\end{align}
Since $\Omega$ is bounded, $\wti{\varphi}_j$ has compact support and hence  
\begin{align}
& (\bbF{\wti \varphi}_j)(\xi)=(2\pi)^{- n/2}\int_{\R^n}{\wti \varphi}_j(x)\overline{\phi(x,\xi)} \, d^nx,
\end{align}
and 
\begin{align} 
\begin{split} 
& \sum_{j \in \bbN} |(\bbF{\wti \varphi}_j)(\xi)|^2 
= (2\pi)^{-n}\sum_{j \in \bbN} \bigg|\int_{\R^n}{\wti \varphi}_j(x)\overline{\phi(x,\xi)} \, d^nx\bigg|^2 \\
& \quad 
=(2\pi)^{-n}\sum_{j \in \bbN} \big|\int_{\Omega}\varphi_j(x)\overline{\phi(x,\xi)} \, d^nx\big|^2 
= (2\pi)^{-n}\|\phi(\, \cdot \,, \xi)\|^2_{L^2(\Omega, d^nx)},    \lb{l3.51} 
\end{split} 
\end{align}
are well-defined. Combining \eqref{3.45} and \eqref{l3.51} one arrives at 
\begin{align}
& N(\lambda; A_{K, \Omega, 2m} (a,b,q)) \leq  \mu^{-1} \int_{\bbR^n} 
\big[\mu -|\xi|^{4m} +\lambda|\xi|^{2m}\big]_+  
\sum_{j \in \bbN} |(\bbF{\wti \varphi}_j)(\xi)|^2 \, d^n \xi\no\\
& \quad = (2\pi)^{-n}  \mu^{-1}\int_{\bbR^n} 
\big[\mu -|\xi|^{4m} +\lambda|\xi|^{2m}\big]_+  
\|\phi(\, \cdot \,, \xi)\|^2_{L^2(\Omega)} \, d^n \xi\no\\
& \quad \leq (2\pi)^{-n}\sup_{\xi\in\R^n}\|\phi(\, \cdot \,, \xi)\|^2_{L^2(\Omega)}  \mu^{-1}\int_{\bbR^n} \lb{l3.52}
\big[\mu -|\xi|^{4m} +\lambda|\xi|^{2m}\big]_+  \, d^n \xi.    
\end{align}
Introducing $\alpha = \lambda^{-2}\mu$, changing 
variables, $\xi = \lambda^{1/(2m)} \eta$, and taking the minimum with respect to $\alpha > 0$, 
proves the bound  
\begin{align}
\begin{split} 
& N(\lambda; A_{K, \Omega, 2m} (a,b,q)) \leq (2 \pi)^{-n} 
\sup_{\xi\in\R^n}\|\phi(\, \cdot \,, \xi)\|^2_{L^2(\Omega)}    \\
&\quad \times \min_{\alpha > 0}
\bigg(\alpha^{-1}\int_{\bbR^n} \big[\alpha - |\eta|^{4m} + |\eta|^{2m}\big]_+ d^n \eta\bigg) \lambda^{n/(2m)} ,  \quad \lambda > 0.    \lb{3.53} 
\end{split} 
\end{align} 
Explicitly computing the minimum over $\alpha > 0$ in \eqref{3.53} yields the 
result \eqref{3.36}. This minimization step is carried out in detail in Appendix \ref{sA}. 
\end{proof}

Next, we also derive an upper bound for the eigenvalue counting function 
of the Friedrichs extension $A_{F,\Omega,2m} (a,b,q)$ of $A_{\Omega,2m} (a,b,q)$. 

\begin{theorem} \lb{t4.4}
Assume Hypothesis \ref{h4.1}. Then the following estimate holds:
\begin{align}
& N(\lambda; A_{F,\Omega,2m} (a,b,q)) \leq \frac {v_n} {(2\pi)^n}	
\bigg(1+\frac{2m}{n}\bigg)^{n/(2m)} 
\sup_{\xi\in\R^n}\|\phi(\, \cdot \,, \xi)\|^2_{L^2(\Omega)} \, \lambda^{n/(2m)},      \no \\ 
& \hspace*{8.7cm} \, \text{ for all } \, \lambda > 0,      \lb{II3.36} 
\end{align} 
with $v_n := \pi^{n/2}/\Gamma((n+2)/2)$ and  $\phi (\, \cdot \, , \, \cdot \,)$ given as in 
Theorem \ref{t4.3}. 
\end{theorem}
\begin{proof}
First, one notices that 
\begin{align}
\begin{split}
& N(\lambda; A_{F,\Omega,2m} (a,b,q)) \leq 
\max \big(\dim\, \big\{f \in \dom(A_{F,\Omega,2m} (a,b,q)) \,\big|   \lb{4.25} \\  
& \hspace*{3cm} 
(f, A_{F,\Omega,2m} (a,b,q)f)_{L^2(\Omega)}- \lambda \|f\|^2_{L^2(\Omega)}< 0\big\}\big),
\end{split}
\end{align}
To further analyze the right--hand side of \eqref{4.25} fix $\lambda \in (0,\infty)$ and introduce the auxiliary operator 
\begin{align} 
\begin{split}
& K_{\Omega,2m, \lambda} (a,b,q) := A_{F,\Omega,2m} (a,b,q) - \lambda I_{\Omega}, \\
& \dom(K_{\Omega,2m,\lambda} (a,b,q)) := \dom(A_{F,\Omega,2m} (a,b,q)). 
\end{split} 
\end{align}
We will study the eigenvalue problem,
\begin{equation} 
K_{\Omega,2m, \lambda} (a,b,q) \varphi_j = \mu_j \varphi_j, \quad 
\varphi_j \in \dom(K_{\Omega,m, \lambda} (a,b,q)), 
\end{equation}
where $\{\varphi_j\}_{j \in \bbN}$ represents an orthonormal basis of eigenfunctions in $L^2(\Omega)$ and for simplicity of notation we repeat the eigenvalues $\mu_j$ of 
$K_{\Omega,2m,\lambda} (a,b,q)$ according to their multiplicity. Since 
$\varphi_j \in W_0^{m}(\Omega)$, their zero-extension to all of $\bbR^n$, 
\begin{equation} 
\wti \varphi_j (x) := \begin{cases} \varphi_j(x), & x \in \Omega, \\ 0, & 
x \in \bbR^n \backslash \Omega,  \end{cases}
\end{equation} 
satisfies 
\begin{equation}  \lb{b4}
\wti \varphi_j \in W^{m} (\bbR^n), \quad 
\partial^{\alpha} \wti \varphi_j = \wti{\partial^{\alpha} \varphi_j}, \quad 0 \leq |\alpha| \leq m.
\end{equation} 
	
Next, given $\mu > 0$, one estimates
\begin{align} 
\begin{split} 
\mu^{-1} \sum_{\substack{j \in \bbN \\ \mu_j < \mu}} (\mu - \mu_j) &\geq 
\mu^{-1} \sum_{\substack{j \in \bbN, \\ \mu_j < 0, \, \mu_j < \mu}} (\mu - \mu_j) \geq 
\mu^{-1} \sum_{\substack{j \in \bbN, \\ \mu_j < 0, \, \mu_j < \mu}} \mu   \\
&= n_-(K_{\Omega, 2m, \lambda} (a,b,q)),   
\end{split} 
\end{align}
where $n_-(K_{\Omega, 2m, \lambda} (a,b,q))$ denotes the number of strictly negative eigenvalues of $K_{\Omega, 2m,\lambda} (a,b,q)$. Then one has
\begin{align}
& N(\lambda; A_{F,\Omega,2m} (a,b,q))   \no \\
& \quad \leq \max \big(\dim\, \big\{f \in \dom(A_{F,\Omega,2m} (a,b,q)) \,\big|    \no \\ 
& \qquad \;
(f, A_{F,\Omega,2m} (a,b,q) f)_{L^2(\Omega)}- \lambda \|f\|^2_{L^2(\Omega)}< 0\big\} \big)  
 \lb{3.44II}  \\ 
& \quad = n_-(K_{\Omega, 2m, \lambda} (a,b,q)) \leq 
\mu^{-1} \sum_{\substack{j \in \bbN \\ \mu_j < \mu}} (\mu - \mu_j) 
= \mu^{-1} \sum_{j \in \bbN} [\mu - \mu_j]_+, \quad \mu > 0.    \no 
\end{align} 
To estimate the right-hand side of \eqref{3.44II} we rewrite 
$(\psi_1, A_{F,\Omega,2m} (a,b,q) \psi_2)_{L^2({\Omega})}$ for 
$\psi_1,\psi_2\in \dom (A_{F,\Omega,2m} (a,b,q))$, as follows
\begin{align}\lb{m16}
& (\psi_1, A_{F,\Omega,2m} (a,b,q) \psi_2)_{L^2({\Omega})}
=Q_{A_{F,\Omega,2m} (a,b,q)}(\psi_1,\psi_2) 
= Q_{\wti A_{2m}(a,b,q)} \big(\wti \psi_1,\wti \psi_2\big)   \no \\
& \quad = \Big(\big(\wti A_{2m} (a,b,q)\big)^{1/2} \wti \psi_1, 
\big(\wti A_{2m} (a,b,q)\big)^{1/2}\wti \psi_2\Big)_{L^2(\bbR^n)},
\end{align}
the second equality in \eqref{m16} following from representations \eqref{m14}, \eqref{m15}. 
Next, we focus on estimating the right-hand side of \eqref{3.44II}. 
\begin{align}
& N(\lambda; A_{F, \Omega, 2m} (a,b,q)) \leq \mu^{-1} \sum_{j \in \bbN} (\mu - \mu_j)_+ 
= \mu^{-1} \sum_{j \in \bbN} \big[(\varphi_j,(\mu - \mu_j) \varphi_j)_{L^2(\Omega)}\big]_+ \no \\   
& \quad = \mu^{-1} \sum_{j \in \bbN} \Big[\mu \|\varphi_j\|_{L^2(\Omega)}^2  
+\lambda \|\varphi_j\|_{L^2(\Omega)}^2 
- (\varphi_j, A_{F,\Omega,2m} (a,b,q) \varphi_j)_{L^2(\Omega)}\Big]_+   \no \\ 
& \quad = \mu^{-1} \sum_{j \in \bbN} \Big[\mu \|\bbF\wti \varphi_j\|_{L^2(\bbR^n)}^2  
+\lambda \|\bbF \wti \varphi_j\|_{L^2(\bbR^n)}^2 
- \||\xi|^m\bbF\wti\varphi_j\|_{L^2(\R^n)}^2\Big]_+   \no \\
& \quad = \mu^{-1} \sum_{j \in \bbN} \bigg[\int_{\bbR^n} 
\big[\mu +\lambda-|\xi|^{2m}\big] |(\bbF{\wti \varphi}_j)(\xi)|^2 \, d^n \xi\bigg]_+   \no \\
& \quad \leq \mu^{-1} \sum_{j \in \bbN} \int_{\bbR^n} 
\big[\mu +\lambda-|\xi|^{2m}\big]_+  
|(\bbF{\wti \varphi}_j)(\xi)|^2 \, d^n \xi    \no \\
& \quad \leq \mu^{-1} \int_{\bbR^n} 
\big[\mu +\lambda-|\xi|^{2m}\big]_+  
\sum_{j \in \bbN} |(\bbF{\wti \varphi}_j)(\xi)|^2 \, d^n \xi.    \lb{II3.45} 
\end{align}
Combining \eqref{l3.51} and \eqref{II3.45} one arrives at 
\begin{align}
& N(\lambda; A_{F,\Omega,2m} (a,b,q)) \leq  \mu^{-1} \int_{\bbR^n} 
\big[\mu +\lambda-|\xi|^{2m}\big]_+  
\sum_{j \in \bbN} |(\bbF{\wti \varphi}_j)(\xi)|^2 \, d^n \xi\no\\
& \quad = (2\pi)^{-n}  \mu^{-1}\int_{\bbR^n} 
\big[\mu +\lambda-|\xi|^{2m}\big]_+  
\|\phi(\, \cdot \,, \xi)\|^2_{L^2(\Omega)} \, d^n \xi\no\\
& \quad \leq (2\pi)^{-n}\sup_{\xi\in\R^n}\|\phi(\, \cdot \,, \xi)\|^2_{L^2(\Omega)}  \mu^{-1}\int_{\bbR^n} \lb{IIl3.52}
\big[\mu +\lambda-|\xi|^{2m}\big]_+  \, d^n \xi. 
\end{align}

Introducing $\alpha = \lambda^{-1}\mu$, changing 
variables, $\xi = \lambda^{1/(2m)} \eta$, and taking the minimum with respect to $\alpha > 0$, 
proves the bound, 
\begin{align} 
\begin{split} 
& N(\lambda; A_{F, \Omega, 2m} (a,b,q)) 
\leq (2 \pi)^{-n} \sup_{\xi\in\R^n}\|\phi(\, \cdot \,, \xi)\|^2_{L^2(\Omega)}    \\
&\quad \times \min_{\alpha > 0}
\bigg(\alpha^{-1}\int_{\bbR^n} \big[\alpha+1- |\eta|^{2m}\big]_+ d^n \eta\bigg) \lambda^{n/(2m)},   
\quad \lambda > 0.    \lb{II3.53} 
\end{split} 
\end{align} 
Denoting 
\begin{equation}\lb{f10}
\cI_F(\alpha):=\alpha^{-1}\int_{\bbR^n} \big[\alpha+1- |\eta|^{2m}\big]_+ d^n \eta,
\end{equation}
one explicitly computes $\cI_F(\alpha)$ and obtains 
\begin{align}
 &\cI_F(\alpha)=\frac{2mv_n}{2m+n}\alpha^{-1}(\alpha+1)^{(2m+n)/(2m)},\lb{f11}\\
 &\cI'_F(\alpha)=\frac{nv_n}{2m+n}(\alpha+1)^{n/(2m)}\alpha^{-2}\left(\alpha-\frac{2m}{n}\right),\lb{f12}\\
 &\min_{\alpha>0} \big(\cI_F(\alpha)\big) = \cI_F(2m/n)=v_n\left(1+\frac{2m}{n}\right)^{n/(2m)}.\lb{f13}
\end{align}
Equation \eqref{f13} together with \eqref{II3.53} yields \eqref{II3.36}.
\end{proof}

\begin{remark} \lb{r4.5}
$(i)$ One notes that whenever the property 
\begin{equation}
\sup_{(x,\xi) \in \Omega \times \bbR^{n}} (|\phi(x,\xi)|) < \infty  
\end{equation} 
has been established, then 
\begin{equation} 
\sup_{\xi \in \bbR^n} \|\phi(\, \cdot \,, \xi)\|^2_{L^2(\Omega)} \leq |\Omega| 
\sup_{(x,\xi) \in \Omega \times \bbR^{n}} \big(|\phi(x,\xi)|^2\big), 
\end{equation}
explicitly exhibits the volume dependence on $\Omega$ of the right-hand sides of \eqref{3.36} 
and \eqref{II3.36}, respectively. We will briefly revisit this in Section \ref{s5}.  \\[1mm]
$(ii)$ Given two self-adjoint operators $A$, $B$ in $\cH$ bounded from below with purely 
discrete spectra such that $A \leq B$ in the sense of quadratic forms, then clearly 
$N(\lambda; B) \leq N(\lambda; A)$, $\lambda \in \bbR$; in addition, 
$N(\lambda; \alpha A) = N(\lambda / \alpha; A)$, $\alpha >0$, $\lambda \in \bbR$. Thus, since 
$a$ is real symmetric, the uniform ellipticity condition \eqref{f3} implies $a \geq \varepsilon_a I_n$, 
and hence $A_{F,\Omega,2}(a,b,q) \geq \varepsilon_a A_{F,\Omega,2}(I_n,b,q)$ assuming 
$\varepsilon_a \in (0,1]$ without loss of generality. Combining this 
with \eqref{2.24} then yields
\begin{align}
\begin{split} 
N(\lambda; A_{K,\Omega,2}(a,b,q)) &\leq N(\lambda; A_{F,\Omega,2}(a,b,q)) \leq 
N(\lambda; \varepsilon_a A_{F,\Omega,2}(I_n,b,q))    \\
&= N(\lambda / \varepsilon_a; A_{F,\Omega,2}(I_n,b,q)), \quad \lambda \in \bbR.  
\end{split} 
\end{align} 
Finally, we note that estimates of the type 
$N(\lambda; A) \leq c_A \lambda^\gamma$ for $A \geq 0$ yield lower bounds for the $j$th 
eigenvalue $\lambda_j(A)$ of the form $\lambda_j (A) \geq d_A j^{1/\gamma}$, clearly 
applicable in the context of \eqref{3.36} and \eqref{II3.36}.   
\hfill $\diamond$
\end{remark}

\begin{remark} \lb{r4.6}
As far as we know, employing the technique of the eigenfunction transform (i.e., the distorted 
Fourier transform) associated with the variable coefficient operator $\wti A_{2m} (a,b,q)$ 
(replacing the standard Fourier transform in connection with the constant coefficient case in \cite{GLMS15}) to derive the results \eqref{1.11} and \eqref{1.12} is new. 

On the other hand, the literature on eigenvalue counting function bounds in connection with arbitrary bounded open sets $\Omega \subset \bbR^n$ (or even open sets $\Omega \subset \bbR^n$ of finite Euclidean volume) is fairly extensive, originating with the seminal work by Birman--Solomyak, Rozenblum, and others. More specifically, starting around 1970, in this context of rough sets 
$\Omega$, Birman and Solomyak pioneered the leading-order Weyl asymptotics and eigenvalue counting function estimates for generalized (linear pencil) eigenvalue problems of the form $A f = \lambda B f$ for elliptic partial differential operators $A$ of order $n_A$ and lower-order differential operators $B$ of order $n_B < n_A$ and obtained great generality of the coeffcients in $A$ and $B$ 
by systematically employing a variational formulation of this generalized eigenvalue problem. The boundary conditions employed are frequently of Dirichlet type, but Neumann and Robin boundary conditions are studied as well. In particular (focusing on the Dirichlet case only), the variational form 
of the problem associated with  
\begin{equation}
\sum_{|\alpha| = |\beta| = m} D^{\alpha} \big(a_{\alpha, \beta}(x) D^{\beta} u\big)(x) 
= \lambda \, p(x) \, u(x), \quad u \in W_0^{m,2}(\Omega),    \lb{4.46}
\end{equation}
with special emphasis on the polyharmonic case, $(- \Delta)^m u = \lambda \, p \, u$, and extensions to 
the situation 
\begin{align}
\begin{split} 
\sum_{|\alpha| = |\beta| = m} D^{\alpha} \big(a_{\alpha, \beta}(x) D^{\beta} u\big)(x) = \lambda 
\sum_{0 \leq |\gamma|, |\delta| \leq m} D^{\gamma} \big(b_{\gamma, \delta}(x) D^{\delta} u\big)(x),&   \\ 
|\gamma| + |\delta| = 2 \ell, \, 0 \leq \ell < m, \; u \in W_0^{m,2}(\Omega),&   \lb{4.47}
\end{split}
\end{align}
including the scenario where $a$, $b$ are block matrices, or $b$ is an appropriate (matrix-valued) measure, were studied in \cite{BS70}--\cite{BS80}, \cite{Ro71}--\cite{Ro76}, \cite[Ch.~5]{RSS94}. 
In particular, the hypotheses on $a_{\alpha, \beta}$ are very general 
($a \in L^1_{loc}(\Omega)^{m \times m}$, $a$ positive definite a.e., 
$a^{-1} \in L^{\alpha}(\Omega)^{m \times m}$ for appropriate $\alpha \geq 1$) permitting a certain weak degeneracy of the ellipticity of the left-hand side in \eqref{4.46}, \eqref{4.47}. The case of the Friedrichs extension for $m=1$ corresponding to $\tau_2(a,b,q)$ was treated in \cite{MR96}. 

Thus, in the case $m=1$, $p(\cdot)=1$, and in some particular higher-order cases, where 
$m > 1$, in the context of $A_{F,\Omega,2m}(a,0,0)$ (i.e., $b=q=0$), there is clearly some overlap 
of our result \eqref{II3.36} with the above results concerning \eqref{4.46}. The same applies to the magnetic field results in \cite{MR96} in connection with $\tau_2(a,b,q)$. Similarly, considering the perturbed buckling problem 
in the form
\begin{equation}
(-\Delta)^{2m} u = \lambda \, (-\Delta)^m u, \quad u \in W_0^{2m,2}(\Omega), 
\end{equation}
there is of course some overlap between our result \eqref{3.36} (actually, the result in \cite{GLMS15}) and the results concerning \eqref{4.47} with $m \in \bbN$, $a=I_n$, $b=q=0$, but since lower-order terms are not explicitly included on the left-hand side of \eqref{4.47}, a direct comparison is difficult. According to G.\ Rozenblum (private communication), the left-hand sides in \eqref{4.46}, \eqref{4.47} can be extended to include also lower-order terms under appropriate hypotheses on the coefficients, but this seems not to have appeared explicitly in print.  

Since we focused on the case of nonconstant coefficients throughout, we did not enter the vast 
literature on eigenvalue counting function estimates in connection with the Laplacian and its 
(fractional) powers. In this context we refer, for instance, to \cite{FLW09}, \cite{HH11}, \cite{YY12}, 
\cite{YY14}, and the extensive literature cited therein.  
\hfill $\diamond$
\end{remark}

Although Weyl asymptotics itself is not the main objective of this paper, we conclude this section with the following observation.

\begin{remark} \lb{r4.7}
The Weyl asymptotics of $N(\, \cdot \,; A_{K, \Omega, 2} (a,b,q))$ in \cite[Sect.~8]{AGMT10} in the 
case of quasi-convex domains and in \cite{BGMM16} in the case of bounded Lipschitz domains 
derived an error bound of the form $\Oh\big(\lambda^{(n - (1/2))/2}\big)$ as $\lambda \to \infty$. If 
one is only interested in the leading-order asymptotics results, combining the spectral equivalence 
of nonzero eigenvalues of $A_{K, \Omega, 2m} (a,b,q))$ to the (generalized) buckling problem 
(cf.\ Lemma \ref{l2.5}), with results by Kozlov \cite{Ko79}--\cite{Ko84}, and taking into account that 
lower-order differential operator perturbations do not influence the leading-order asymptotics of 
$N(\, \cdot \, ;A_{K, \Omega, 2m} (a,b,q))$ (cf.\ \cite[Lemmas~1.3, 1.4]{BS72}) imply 
\begin{align}
& N(\lambda; A_{K,\Omega,2m} (a,b,q))   \no \\ 
& \quad \underset{\lambda \to \infty}{=} 
\frac{1}{n (2\pi)^n} \bigg(\int_{\Omega}d^nx\int_{|\xi|=1}d\omega_{n-1}(\xi)(\xi,a(x) \, \xi)^{-\frac{n}{2}}_{\R^n} \bigg)
 \, \lambda^{n/(2m)} + \oh\big( \lambda^{n/(2m)}\big)     \no \\
& \quad \underset{\lambda \to \infty}{=} \frac{v_n}{(2\pi)^n} \bigg( \int_{\Omega}d^nx \, 
(\det a(x))^{-1/2} \bigg)
 \, \lambda^{n/(2m)}   
+ \oh\big( \lambda^{n/(2m)}\big),    \lb{4.49}
\end{align}
for any bounded open set $\Omega \subset \bbR^n$. Here $d\omega_{n-1}$ denotes the surface measure on the unit sphere $S^{n-1}=\{\xi\in\bbR^n\,|\,|\xi|=1\}$ in $\bbR^n$. Of course, the same 
leading-order asymptotics applies to $N(\, \cdot \, ; A_{F,\Omega,2m} (a,b,q))$. 

Since $N(\lambda; A) \underset{\lambda \to \infty}{=} c(A) \lambda^{\alpha}$ is equivalent to 
$\lambda_j(A) \underset{j \to \infty}{=} (j/c(A))^{1 / \alpha}$, relation \eqref{4.49} yields the corresponding result for the eigenvalues of $A_{K,\Omega,2m} (a,b,q)$ and $A_{F,\Omega,2m} (a,b,q)$. 
\hfill $\diamond$
\end{remark}

\section{Illustrations} \lb{s5}
 
To demonstrate why we expect Hypothesis \ref{h4.1} to hold under Hypothesis \ref{h3.1} alone in the case $a = I_n$ (with the obvious exception of zero-energy resonances and eigenvalues, which generically will be absent), we discuss three exceedingly complex scenarios in this section.

We start with the most elementary case which nevertheless served as the guiding motivation for 
this paper:   

\begin{example} \lb{e5.1} Let $a:=I_n$, $n \in \bbN$, $b= q = 0$, then the operator $\bbF$ from Theorem \ref{h4.1} 
is the standard Fourier transform in $L^2(\bbR^n)$, and $\phi(\xi,x)=e^{i\xi \cdot x}$, $(\xi, x) \in \bbR^{2n}$. Thus, Hypothesis \ref{h4.1} obviously holds for $\wti A_2 (I_n,0,0) = H_0$, and 
\begin{equation}\lb{aa1}
\sup_{\xi\in \R^n}\|\phi(\, \cdot \,, \xi)\|_{L^2(\Omega)}^2 = |\Omega|. 
\end{equation}
\end{example}

In this rather special case the estimate for the eigenvalue counting function 
$N(\lambda; - \Delta_{K,\Omega})$ was previously obtained in \cite{GLMS15}, while that 
of $N(\lambda; - \Delta_{D,\Omega})$ was derived in \cite{La97}.

Next, we turn to Schr\"odinger operators in $L^2(\bbR^n)$. 

\begin{example} \lb{e5.2} Assume that $a = I_n$, $b=0$, and 
$0\leq q\in L^{\infty}(\R^n)$, $\supp (q)$ compact. In addition, suppose that zero is neither an eigenvalue nor a resonance  of $\wti A_2(I_n,b,q)$ $($cf.\ \cite{EGS09}$)$. Then 
Hypothesis \ref{h4.1} holds.

In addition, in the special case $n=3$, there exists $C(q) \in (0,\infty)$ such that 
\begin{equation}
\sup_{(x,\xi) \in \R^6}|\phi(x,\xi)| \leq C({q}).   \lb{5.2}
\end{equation}  
\end{example}

Indeed, the absence of strictly positive eigenvalues of $\wti A_2 (I_n,0,q)$ was established by 
Kato \cite{Ka59} (see also \cite{Si67}), and the existence of the distorted Fourier transform $\bbF$ and hence an eigenfunction 
transform was established by Ikebe \cite[Theorem~5]{Ik60} for $n=3$ and Thoe \cite[Sect.~4]{Th67} 
for $n \geq 4$, and Alsholm and Schmidt \cite{AS71} for $n\geq 3$ (see also 
\cite[Theorem~XI.41]{RS79}, \cite[Theorems~XIII.33 and XIII.58]{RS78}, \cite{RT15}, 
\cite[Sect.~V.4]{Si71}), implying, in particular, that 
\begin{align}
\begin{split} 
& \sigma \big(\wti A_2 (I_n,0,q)\big)=\sigma_{ac}\big(\wti A_2(I_n,0,q)\big)=[0,\infty),  \lb{5.3} \\
&\sigma_{sc}\big(\wti A_2(I_n,0,q)\big) = \sigma_{p}\big(\wti A_2(I_n,0,q)\big) \cap (0,\infty) 
= \emptyset.
\end{split}
\end{align}
Moreover, it is shown in \cite{Ik60} and \cite{Th67} that for all $R>0$, 
\begin{equation}\lb{aa2}
	\sup_{\xi\in B_n(0;R), \, x\in\R^n}|\phi(x,\xi)|=:c({q,R}) < \infty. 
\end{equation}  
Thus we will focus on proving that 
\begin{equation}\lb{5.5}
\sup_{\xi\in \R^n} \|\phi(\, \cdot \,,\xi)\|_{L^{2}(\Omega)}<\infty,    
\end{equation}
and in the special case $n=3$ that for sufficiently large $R > 0$, 
\begin{equation}\lb{aa3}
\sup_{\xi\in \R^3 \backslash B_3(0;R), \, x\in\R^3}|\phi(x,\xi)|=:C({q,R}) < \infty. 
\end{equation}  
Clearly, estimates \eqref{aa2} and \eqref{aa3} imply \eqref{5.2}. 

The distorted plane waves $\phi(\, \cdot \, , \, \cdot \,)$ can be chosen as one of 
$\phi_+(\, \cdot \, , \, \cdot \,)$  or $\phi_-(\, \cdot \, , \, \cdot \,)$, which are defined as solutions of the following Lippmann--Schwinger integral equation, 
\begin{equation}\lb{aa4}
\phi_{\pm}(x,\xi)=e^{i\xi\cdot x} 
- \int_{\bbR^n} G_n \big(|\xi|^2 \pm i 0;x,y\big) q(y)\phi_{\pm}(y,\xi) \, d^ny, 
\quad (x,\xi) \in \bbR^{2n}, 
\end{equation} 
where
\begin{align}
& G_n(z;x,y) = \begin{cases} \f{i}{4} \Big(\f{2\pi |x-y|}{z^{1/2}}\Big)^{(2-n)/2} 
H^{(1)}_{(n-2)/2}\big(z^{1/2}|x-y|\big), & n\ge 2, \; z\in\bbC\backslash\{0\}, \\
\f{-1}{2\pi} \ln(|x-y|), & n=2, \; z=0, \\
\f{1}{(n-2) \omega_{n-1}} |x-y|^{2-n}, & n \ge 3, \; z=0, 
\end{cases}    \no \\
& \hspace*{6.15cm}   \Im\big(z^{1/2}\big)\geq 0, \; x, y \in\bbR^n, \, x \neq y, 
\end{align}
represents the fundamental solution of the Helmholtz equation $(-\Delta -z)\psi(z;\dott) =0$ in 
$\bbR^n$, that is, the Green's function of the $n$-dimensional Laplacian, $n\in\bbN$, $n\ge 2$. 
Here $H^{(1)}_{\nu}(\dott)$ denotes the Hankel function of the first kind 
with index $\nu\geq 0$ (cf.\ \cite[Sect.\ 9.1]{AS72}) and 
$\omega_{n-1}=2\pi^{n/2}/\Gamma(n/2)$ ($\Gamma(\dott)$ the Gamma function, 
cf.\ \cite[Sect.\ 6.1]{AS72}) represents the volume of the unit sphere 
$S^{n-1}$ in $\bbR^n$.   
For simplicity we focus on $n \geq 3$ for the rest of this example, but note that the cases $n=1,2$ 
can be treated exactly along the same lines (see, e.g., the results in \cite{BGD88}--\cite{Ch84}).
  
Multiplying both sides of this equation by the weight $w > 0$ satisfying 
\begin{align} 
\begin{split}
& w \in C^{\infty}(\bbR^n), \quad 0 < w \leq 1, \quad w(x):= \begin{cases} 1, & 0 \leq |x| \leq R, \\
\exp(-|x|^2), & |x| \geq 2R, \end{cases} \\ 
& \Omega \subset B_n(0; R), 
\end{split} 
\end{align} 
for some $R>0$, \eqref{aa4} can be written as follows
\begin{align}\lb{aa5}
\begin{split} 
\Phi_{\pm}(x,\xi)=\Phi_0(x,\xi) - \int_{\bbR^n} w(x)G_n \big(|\xi|^2 \pm i 0;x,y\big) 
w(y) \frac{q(y)}{w^{2}(y)} \Phi_{\pm}(y,\xi) \, d^3y,&    \\
(x,\xi) \in \bbR^{2n},& 
\end{split} 
\end{align} 
where 
\begin{equation}\lb{aa15}
\Phi_{\pm}(x,\xi):= w(x) \phi_{\pm}(x,\xi), \quad  \Phi_0(x,\xi):= w(x) e^{i\xi\cdot x}, 
\quad (x,\xi) \in \bbR^{2n}. 
\end{equation} 
In this form \eqref{aa5} becomes an integral equation in $L^2(\R^n)$ since 
$\Phi_0(\, \cdot \,,\xi)\in L^2(\R^n)$. In fact, \eqref{aa5} will be viewed in $L^2(\bbR^n)$ as 
\begin{equation}\lb{aa6}
\Phi_{\pm}(\, \cdot \,,\xi) = \Phi_0(\, \cdot \,,\xi) + K_{\pm}(\xi) M_{q/w^2} 
\Phi_{\pm}(\, \cdot \,,\xi), \quad \xi \in \bbR^n, 
\end{equation} 
or equivalently, as 
\begin{equation}\lb{aa7}
[I_{L^2(\R^n,d^n x)} - K_{\pm}(\xi) M_{q/w^2}]\Phi_{\pm}(\, \cdot \,,\xi) 
= \Phi_0(\, \cdot \,,\xi), \quad \xi \in \bbR^n, 
\end{equation}
where we introduced the Birman--Schwinger-type operator $K_{\pm}(\xi)$, $\xi \in \bbR^n$, in 
$L^2(\bbR^n)$, 
\begin{align} 
& K_{\pm}(\xi) \in \cB\big(L^2(\R^n)\big),    \no \\
&(K_{\pm}(\xi)f)(x):= - \int_{\bbR^n} w(x) G_n\big(|\xi|^2 \pm i 0;x,y\big) w(y) f(y,\xi) \, d^n y,   
 \lb{aa8} \\ 
& \hspace*{5.4cm} f \in L^2(\bbR^n), \; (x,\xi) \in \bbR^{2n},   \no 
\end{align} 
and the operator of multiplication by the function $q/w^2$, $M_{q/w^2}$ in $L^2(\bbR^n)$, 
\begin{equation}
M_{q/w^2} \in \cB\big(L^2(\R^n)\big), \quad 
(M_{q/w^2} f)(x):= q(x) w(x)^{-2} f(x), \quad f \in L^2(\bbR^n), \; x \in \bbR^n.    \lb{aa9}
\end{equation} 
One recalls from \cite[Sect.~V.4]{Si71} for $n =3$ and \cite{Fa71} for $n \geq 3$ (the case $n=2$ 
being analogous) that 
\begin{equation} 
\|K_{\pm}(\xi)\|_{\cB(L^2(\R^n))}\underset{|\xi| \to \infty}{\longrightarrow} 0,
\end{equation}  
and hence, 
\begin{align}
&\|\Phi_{\pm}(\, \cdot \,,\xi) - \Phi_0(\, \cdot \,,\xi)\|_{L^2(\R^n)}     \no \\
& \quad = \big\|\left(I_{L^2(\R^n)}-(I_{L^2(\R^n)} - K_{\pm}(\xi) M_{q/w^2}))^{-1}\right) 
\Phi_0(\, \cdot \,,\xi)\big\|_{L^2(\R^n)}    \no\\
&\quad\leq \|w(\cdot)\|_{L^2(\R^n)} \big\|I_{L^2(\R^n)}-(I_{L^2(\R^n)} 
- K_{\pm}(\xi) M_{q/w^2}))^{-1}\big\|_{\cB(L^2(\R^n))}      \no \\ 
& \, \underset{|\xi|\rightarrow \infty }{=}\oh(1),    
\end{align}
implying, 
\begin{equation}\lb{aa11}
\|\Phi_{\pm}(\, \cdot \,,\xi)\|_{L^2(\R^n)} \underset{|\xi|\rightarrow \infty}{=} \Oh(1), 
\end{equation}
and hence \eqref{5.5}.

In the special case $n=3$, where 
\begin{equation} 
G_3(z;x,y) = (4 \pi |x-y|)^{-1} e^{i z^{1/2} |x-y|}, \quad \Im\big(z^{1/2}\big)\geq 0, \; 
x, y \in \bbR^3, \, x \neq y, 
\end{equation} 
one can easily go one step further: Using the Cauchy--Schwarz inequality, \eqref{aa11}, and 
the fact that $q$ has compact support, one estimates the second term in \eqref{aa5} as follows, 
\begin{align}
&\bigg|\int_{\bbR^3}\frac{w(x)e^{\pm i|\xi||x-y|}w(y)}{4 \pi |x-y|}\frac{q(y)}{w^{2}(y)} 
\Phi_{\pm}(y,\xi) \, d^3 y\bigg|  
\no \\
&\quad \leq (4 \pi)^{-1} w(x)\int_{\supp (q)}\frac{w(y)}{|x-y|}\frac{q(y)}{w^{2}(y)} 
|\Phi_{\pm}(y,\xi)| \, d^3 y   \no \\
& \quad \leq (4 \pi)^{-1} w(x)\|qw^{-2}\|_{L^{\infty}(\R^3)} 
\bigg(\int_{\supp (q)}\frac{w^2(y)}{|x-y|^2} \, d^n y\bigg)^{1/2} 
\|\Phi_{\pm}(\, \cdot \,,\xi)\|_{L^2(\R^3)}  \no \\
& \, \underset{ |\xi|\rightarrow \infty}{=} w(x)\Oh(1),  \quad x \in \bbR^3,  \lb{aa14}
\end{align}
with the $\Oh(1)$-term bounded uniformly in $(x, \xi) \in\R^{6}$. Combining \eqref{aa15}, 
\eqref{aa6}, and \eqref{aa14} one obtains 
\begin{equation} 
\sup_{x\in\R^3}|\phi_{\pm}(x,\xi)| \underset{|\xi| \to \infty}{=} \Oh(1),   \lb{5.17} 
\end{equation}
proving \eqref{aa2} since $\phi_{\pm}$ is continuous on $\bbR^{6}$ (see, e.g., 
\cite[Sect.~4]{Ik60}, \cite[Sect.~3]{Th67}). \hfill $\square$

\begin{example} \lb{e5.3} 
Assume that $n\in\bbN$, $a=I_n,\ b\in \big[W^{1,\infty}(\R^n)\big]^n$, $\supp (b)$ compact, 
$0\leq q\in L^{\infty}(\R^n)$, $\supp (q)$ compact. In addition, suppose that zero is neither an eigenvalue nor a resonance  of $\wti A_2(I_n,b,q)$ $($cf.\ \cite{EGS09}$)$. Then Hypothesis \ref{h4.1} holds.
\end{example}

We start verifying this claim by noting that under these assumptions on $a,b,q$, $\wti A_2(I_n,b,q)$ 
has empty singular continuous spectrum and no strictly positive eigenvalues, see, for instance, 
Erdogan, Goldberg, and Schlag \cite{EGS08}, \cite{EGS09}, Ikebe and Sait{\=o} \cite{IS72}, (see also, \cite{AZ15}, \cite{BA10}, \cite{Fa09}, \cite{Ga15}, \cite{KT09}, \cite{St90}); in particular, the analog of \eqref{5.3} holds for $\wti A_2(I_n,b,q)$.

Next, we recall the unperturbed operator $H_0:=-\Delta$, $\dom (H_0)=W^{2,2}(\R^n)$, and introduce the first-order perturbation term, 
\begin{equation}\lb{e1n}
L_1 f =2i \sum_{k=1}^{n} b_k\partial_k f + (i \, \text{div}(b)+|b|^2+q)f,   \quad  
f \in \dom(L_1) = W^{1,2}(\bbR^n).	
\end{equation} 
We denote the distorted plane waves associated with $\wti A_2 (I_n,b,q)$ by 
$\phi(\, \cdot \,,\, \cdot \,)$, and abbreviate
\begin{equation}
	\phi_0(x,\xi):=e^{i\xi\cdot x}, \quad (x,\xi) \in \R^{2n}.
\end{equation} 

In the following we will show that
\begin{equation}\lb{kk11n}
\sup_{\xi\in \R^n} \|\phi(\, \cdot \,,\xi)\|_{L^{2}(\Omega)}<\infty.   
\end{equation}
To this end, we employ \cite[Theorem 1.2]{EGS09} (see also \cite[Theorem~2]{EGS08}) with 
$\alpha=0$, $\sigma=1$ and infer 
\begin{equation}\lb{egs}
K:=\sup_{|\xi| \geq 0} \, \big(\langle |\xi| \rangle \big\|\langle \, \cdot \, \rangle^{-2} 
\big(\wti A_2 (I_n,b,q)-(|\xi|^2 \pm i0)\big)^{-1}\langle \, \cdot \, \rangle^{-2}\big\|_{\cB(L^2(\R^n))}\big) 
< \infty,
\end{equation}
abbreviating $\langle \, \cdot \, \rangle:=\big[1+(\, \cdot \,)^2\big]^{1/2}$.

The distorted plane wave $\phi(\, \cdot \,, \, \cdot \,)$ can again be chosen as one of 
$\phi_+(\, \cdot \,, \, \cdot \,)$ or $\phi_-(\, \cdot \,, \, \cdot \,)$ and be decomposed in the form 
\begin{equation}\lb{aaa3}
\phi_{\pm}(x,\xi) = \phi_0(x,\xi) + \psi_{\pm}(x,\xi), \quad (x,\xi) \in \bbR^{2n}, 
\end{equation}
where
\begin{equation}\lb{aaa4}
\psi_{\pm}(x,\xi) := - \big(\big(\wti A_2 (I_n,b,q)-(|\xi|^2 \pm i0)\big)^{-1} (L_1 \phi_0)\big)(x,\xi), 
\quad (x,\xi) \in \bbR^{2n}.  
\end{equation}
(In this context we recall that 
\begin{align} 
\begin{split} 
|\xi|^2 \phi_{\pm}(x,\xi) &= \big(\wti A_2 (I_n,b,q) \phi_{\pm}\big)(x,\xi)   \\
&= |\xi|^2 \phi_0(x,\xi) + (L_1 \phi_0)(x,\xi) + \big(\wti A_2 (I_n,b,q) \psi_{\pm}\big)(x,\xi), 
\end{split} 
\end{align}
or equivalently, 
\begin{equation}
- (L_1 \phi_0)(x,\xi) = \big(\big(\wti A_2(I_n,b,q)) - |\xi|^2\big) \psi_{\pm}\big)(x,\xi), 
\end{equation}
in the sense of distributions, illustrating \eqref{aaa4}.)

One then infers  
\begin{align}
&\|\psi_{\pm}(\, \cdot \,, \xi)\|_{L^2(\Omega)}   \no \\
& \quad = \big\|\chi_{\Omega}\langle \, \cdot \, \rangle^{2}\langle \, \cdot \, \rangle^{-2} 
\big(\wti A_2 (I_n,b,q)-(|\xi|^2 \pm i0)\big)^{-1} \langle \, \cdot \, \rangle^{-2} 
\langle \, \cdot \, \rangle^{2} (L_1 \phi_0)\big\|_{L^2(\R^n)}   \no \\
&\quad\leq \big\|\chi_{\Omega}\langle \, \cdot \, \rangle^{2}\big\|_{L^{\infty}(\R^n)} \, 
\big\|\langle \, \cdot \,\rangle^{-2} \big(\wti A_2 (I_n,b,q)-(|\xi|^2 \pm i0)\big)^{-1} 
\langle \, \cdot \, \rangle^{-2}\big\|_{\cB(L^2(\R^n))}     \no \\
&\qquad \times \big\|\langle \, \cdot \, \rangle^{2} (L_1 \phi_0)\big\|_{L^2(\R^n)}.	 \lb{525new} 
\end{align}
Employing \eqref{egs}, the fact that $\Omega$ is bounded, and that the coefficients of $L_1$ have 
compact support (cf.\ \eqref{e1n}), one concludes 
\begin{align}
&\big\|\langle \, \cdot \,\rangle^{-2} \big(\wti A_2 (I_n,b,q)-(|\xi|^2 \pm i0)\big)^{-1} 
\langle \, \cdot \, \rangle^{-2}\big\|_{\cB(L^2(\R^n))} \big\|\langle \, \cdot \, \rangle^{2} (L_1 \phi_0)\big\|_{L^2(\R^n)}     \no \\
&\quad \leq K\langle \, |\xi| \, \rangle^{-1}\big\|\langle \, \cdot \, \rangle^2\phi_0(\, \cdot\, ,\xi)
\big( -2b\cdot\xi + i \, \text{div}(b)+|b|^2+q\big)\big\|_{L^2(\R^n)}\no\\
&\quad \leq 2K|\xi|\langle \, |\xi| \, \rangle^{-1} \big\|\langle \, \cdot \, \rangle^2 b\big\|_{[L^2(\R^n)]^n} 
+ K\langle \, |\xi| \, \rangle^{-1} \big\|\langle \, \cdot \, \rangle^2\left(i \, \text{div}(b)+|b|^2+q\right)\big\|_{L^2(\R^n)} \no\\
&\, \underset{|\xi| \to \infty}{=} \Oh(1).     \lb{526new}
\end{align}
Combining \eqref{525new} and \eqref{526new} one obtains the required estimate \eqref{kk11n}. 
\hfill $\square$

\appendix
\section{A Minimization Problem} 
\lb{sA}
\renewcommand{\theequation}{A.\arabic{equation}}
\renewcommand{\thetheorem}{A.\arabic{theorem}}
\setcounter{theorem}{0} \setcounter{equation}{0}

In this appendix we carry out the explicit minimization in $\alpha$ for $\alpha > 0$ 
of the integral 
\begin{equation}
\cI_K(\alpha) := \alpha^{-1}\int_{\bbR^n} \big[\alpha - |\eta|^{4m} + |\eta|^{2m}\big]_+ \, d^n \eta. 
\end{equation}   
Since the integral is only over the region of $n$-space where $\alpha - |\eta|^{4m} + |\eta|^{2m}$ 
is positive, and this function is radial, our problem immediately reduces to the minimization of 
$\alpha^{-1}$ times a 
radial integral in $r=|\eta|.$  Since the function $r^{4m}-r^{2m}=r^{2m}(r^{2m}-1)$ is 
negative on $0< r <1$ and is positive and increasing for $r > 1,$ for $\alpha > 0$ 
the relation $\alpha=r^{4m}-r^{2m}$ implicitly determines a unique value $r_\alpha>1,$ 
with $r_\alpha^{2m}$ given explicitly by 
\begin{equation}
r_\alpha^{2m}=\dfrac12+\Big(\alpha+\dfrac14 \Big)^{1/2}.  \lb{A.1}
\end{equation} 
It is clear that the value of $r_\alpha$ is a strictly increasing function of $\alpha$ and runs from 
1 to $\infty$ as $\alpha$ runs from 0 to $\infty.$  

By the reductions mentioned above, one obtains  
\begin{equation}
\cI_K(\alpha) 
= n v_n \alpha^{-1} \int_0^{r_\alpha} \, [\alpha + r^{2m}-r^{4m}] \, r^{n-1} \, dr,   \lb{A.2} 
\end{equation}
where $v_n$ is the volume of the ball of unit radius in $\bbR^n$ as mentioned with \eqref{3.36}.  Since 
the $v_n$ here is included explicitly in \eqref{3.36}, to prove \eqref{3.36}, in what remains we will show 
that the function $f_{n,m}(\alpha)$ defined by 
\begin{equation}
f_{n,m}(\alpha) := n \alpha^{-1} \int_0^{r_\alpha} \, [\alpha + r^{2m}-r^{4m}] \, r^{n-1} \, dr  \lb{A.3}
\end{equation}
has minimum given by 
\begin{equation}
\wti{f}_{n,m} := \Big(1+\dfrac{2m}{n+2m} \Big)^{n/(2m)},  \quad m, n \in \bbN.   \lb{A.4}
\end{equation}

By integrating \eqref{A.3}, it is easy to see that 
\begin{equation} 
f_{n,m}(\alpha) := n \alpha^{-1}  \bigg[\dfrac{\alpha r_\alpha^n}{n}+\dfrac{r_\alpha^{n+2m}}{n+2m}-\dfrac{r_\alpha^{n+4m}}{n+4m} \bigg].   \lb{A.5} 
\end{equation}
Replacing the explicit $\alpha$ appearing inside the square brackets here 
using $\alpha=r_\alpha^{4m}-r_\alpha^{2m}$ and simplifying, one finds  
\begin{equation}
\alpha \, f_{n,m}(\alpha)=\dfrac{4m \, r_\alpha^{n+4m}}{n+4m}-\dfrac{2m \, r_\alpha^{n+2m}}{n+2m}.   \lb{A.6}
\end{equation} 
We shall have need of this expression shortly.  

Next, some further properties of $f_{n,m}$ and its derivative will be developed. One has  
\begin{equation} 
\alpha \, f_{n,m}(\alpha)=n \int_0^{r_\alpha} \, [\alpha+r^{2m}-r^{4m}] \, r^{n-1} \, dr,    \lb{A.7} 
\end{equation} 
and therefore, by Leibniz's rule,  
\begin{align}
[\alpha \, f_{n,m}(\alpha)]' =& n [\alpha+r_\alpha^{2m}-r_\alpha^{4m}] \, r_\alpha^{n-1} \, r'_\alpha + n \int_0^{r_\alpha} \, r^{n-1} \, dr = r_\alpha^n,  \lb{A.8} 
\end{align}
with the simplification in the last step occurring due to the implicit relation defining $r_\alpha.$  
From \eqref{A.8} it follows that 
\begin{equation}
\alpha \, f_{n,m}'(\alpha)=r_\alpha^n-f_{n,m}(\alpha),   \lb{A.9} 
\end{equation}
and hence, using \eqref{A.6} and $\alpha=r_\alpha^{4m}-r_\alpha^{2m}$, that 
\begin{align}
\alpha^2 f'_{n,m}(\alpha)=&\alpha \, r_\alpha^n - \bigg[\dfrac{4m \, r_\alpha^{n+4m}}{n+4m}-\dfrac{2m \, r_\alpha^{n+2m}}{n+2m} \bigg]  \no \\
=& n \bigg(r_\alpha^{2m}- \dfrac{n+4m}{n+2m} \bigg) \, \dfrac{r_\alpha^{n+2m}}{n+4m}.  \lb{A.10} 
\end{align} 
It is now clear that $f_{n,m}(\alpha)$ has a minimum on $\alpha \in (0,\infty),$ and that it occurs at 
\begin{equation}
r_\alpha^{2m}=\dfrac{n+4m}{n+2m} := \wti{r}_\alpha^{2m}   \lb{A.11} 
\end{equation}
(one notes that this value is clearly larger than 1, and hence corresponds to an $\alpha > 0$).  
The corresponding value of $\alpha,$ denoted by $\wti{\alpha},$ may then be computed as 
\begin{equation}
\wti{\alpha}=\wti{r}_\alpha^{2m}(\wti{r}_\alpha^{2m}-1)=\dfrac{n+4m}{n+2m} \dfrac{2m}{n+2m}=\dfrac{2m(n+4m)}{(n+2m)^2}.   \lb{A.12} 
\end{equation} 

Finally one computes $\wti{f}_{n,m}$ using \eqref{A.6}, \eqref{A.11}, and \eqref{A.12}, which leads to 
\begin{equation}
\wti{f}_{n,m}=f_{n,m}(\wti{\alpha})=\Big(\dfrac{n+4m}{n+2m} \Big)^{n/(2m)}  
= \Big(1+\dfrac{2m}{n+2m} \Big)^{n/(2m)},  \quad m, n \in \bbN,    \lb{A.13} 
\end{equation}
in accordance with our statement above.  This completes the proof of Theorem \ref{t4.3}.  

We conclude with some remarks comparing the constant $\wti{f}_{n,m}$ found here with the corresponding constants 
$\wti{g}_{n,m}$ (our notation) found by Laptev in \cite{La97} (the comparison is most apt if we restrict our attention to the 
case of the Laplacian (i.e., $a = I_n$, $b=q=0$), as that is the main case considered by Laptev \cite{La97}).  Laptev's $\wti{g}_{n,m}$ are given by 
\begin{equation}
\wti{g}_{n,m}=\Big(1+\dfrac{2m}{n} \Big)^{n/(2m)}, \quad m, n \in \bbN.   \lb{A.14} 
\end{equation}
It is clear from these expressions that 
\begin{equation}
\wti{f}_{n,m}<\wti{g}_{n,m}, \quad m, n \in \bbN.    \lb{A.15}
\end{equation} 
This shows that the bound given in Theorem \ref{t4.3} 
is always better than the bound \eqref{2.24} combined with the earlier work of Laptev \cite{La97}. 
Of course in the large $n$ limit (for fixed $m$) both constants become arbitrarily close, since 
the limit of either $\wti{g}_{n,m}$ or $\wti{f}_{n,m}$ as $ n \to \infty$ is $e\approx 2.71828.$  
On the other hand, in the large $m$ limit (with $n$ fixed) both constants go to 1 from above (with 1 being the best possible value of the constant that could be obtained in our upper bounds, at least 
in the case of the Laplacian).  

In fact, it is generally true that 
\begin{equation}
1 < \wti{f}_{n,m} < \wti{g}_{n,m} < e, \quad m, n \in \bbN,   \lb{A.16}
\end{equation} 
that is, that 
\begin{equation} 
1 < (1+2m/(n+2m))^{n/2m} < (1 + (2m/n))^{n/(2m)} < e, \quad m, n \in \bbN,  \lb{A.17} 
\end{equation} 
with 1 and $e$ being the best possible lower and upper bounds for both $\wti{f}_{n,m}$ and 
$\wti{g}_{n,m}$ for all $m, n \in \bbN$.  These claims can be proved using 
elementary calculus by focusing on the functions $G(x):=(\ln(1+x))/x$ and $F(x):=(\ln(1+x/(1+x)))/x$ for $x>0$ (note that with the 
identification $x=2m/n$ these are the logarithms of $\wti{g}_{n,m}$ and $\wti{f}_{n,m},$ respectively, and that all $x>0$ can be 
approximated arbitrarily closely by such ratios for $m, n \in \bbN$).  In fact, one can show that the functions $G(x)$ 
and $F(x)$ are both strictly decreasing on $(0,\infty),$ with limiting value $1$ as $x \to 0^+,$ and with limiting value $0$ as $x \to \infty.$  This implies, in particular, that in all upper bound formulas 
for counting functions $N(\cdot)$ in this paper the bound would continue to hold (as a strict inequality) if the constant represented by $(1+2m/(n+2m))^{n/2m}$ were replaced by the value $e.$  

\medskip

\noindent
{\bf Acknowledgments.} We are indebted to Tanya Christiansen for discussions and to Grigori Rozenblum and Wilhelm Schlag for very helpful correspondence. 

 
\end{document}